\documentclass[11pt]{article}
\usepackage[centertags]{amsmath}
\usepackage{amsfonts}
\usepackage{amssymb}
\usepackage{amsthm}
\usepackage{mathrsfs}
\usepackage{graphicx}
\usepackage{psfrag}

\usepackage{hyperref}

\usepackage{cleveref}

\usepackage{color}

\usepackage{authblk}

\usepackage{tikz}

\usepackage{enumerate}
\usepackage{enumitem}

\usepackage{xspace}

\usepackage{float}

\usepackage{afterpage}

\providecommand{\keywords}[1]{\textbf{\textit{Keywords---}} #1}

\newcommand{\mo}{\mathfrak M}
\newcommand{\Tn}{{\cal T}(n)}
\newcommand{\mT}{{\mathfrak T}}

\newcommand{\U}{{\mathcal{U}}}

\newcommand{\Pcal}{{\mathcal{P}}}

\newcommand{\F}{{\mathfrak{F}}}
\newcommand{\M}{{\mathfrak{M}}}
\newcommand{\mon}{\textsf{MR}\xspace}
\newcommand{\N}{{\mathfrak{N}}}
\newcommand{\f}[1]{\mathfrak{#1}}

\newcommand{\NNIL}{{\sf{NNIL}}\xspace}

\newcommand{\MR}{\textsf{MR}\xspace}

\theoremstyle{definition}\newtheorem{theorem}{Theorem}[section]
\theoremstyle{definition}\newtheorem{definition}[theorem]{Definition}

\theoremstyle{definition}\newtheorem{corollary}[theorem]{Corollary}
\theoremstyle{definition}
\theoremstyle{definition}\newtheorem{lemma}[theorem]{Lemma}
\theoremstyle{definition}
\theoremstyle{definition}\newtheorem{proposition}[theorem]{Proposition}
\theoremstyle{definition}\newtheorem{fact}[theorem]{Fact}

\newenvironment{claim}[1][]{%
\par\addvspace{6pt}
\noindent
\textbf{Claim#1:}
\noindent
}
{\par\addvspace{4pt}}

\newenvironment{proofclaim}[1][]{%
\par
\noindent
\emph{Proof of Claim#1.}
\noindent
}
{\hfill$\dashv$\par\addvspace{6pt}}

\newcommand{\todod}[1]{}
\newcommand{\todof}[1]{}
\newcommand{\todoj}[1]{}

\begin{document}

\title{\NNIL-formulas revisited:\\
\Large{universal models and finite model property}}
\author[$*$]{Julia Ilin}
\author[$*$]{Dick de Jongh}
\author[$\dagger$]{Fan Yang}
\affil[$*$]{\small Institute of Logic, Language and Computation, University of Amsterdam. P.O. Box 94242, 1090 GE Amsterdam, The Netherlands, email:  ilin.juli@gmail.com, d.h.j.dejongh@uva.nl}
\affil[$\dagger$]{\small Department of Mathematics and Statistics, University of Helsinki, PL 68 (Pietari Kalmin katu 5), 00014 University of Helsinki, Finland, email: fan.yang.c@gmail.com}
\date{}

\maketitle\maketitle

\begin{abstract}
\NNIL-formulas, introduced by Visser  in 1983-1984 in a study of $\Sigma_1$-subsitutions in Heyting Arithmetic, are intuitionistic propositional formulas that does not allow nesting of implication to the left. The first 
results about these formulas were obtained in a paper of 1995 by Visser et al. 
\todof{$\leftarrow$I changed ``main results" to ``first results"}
In particular, it was shown that \NNIL-formulas are exactly the formulas preserved under taking submodels of Kripke models. Recently Bezhanishvili and de Jongh observed that \NNIL-formulas are also reflected by color-preserving monotonic maps of Kripke models. \todof{I added ``color-preserving"} In the present paper,
we first show how this observation leads to the conclusion that NNIL-formulas are preserved by arbitrary substructures not necessarily satisfying the topo-subframe condition.\todod{added sentence about substructures} Then we apply it to construct universal models for \NNIL.
It follows from the properties of these universal models that \NNIL-formulas are also exactly the formulas that are reflected by color-preserving monotonic maps. By using the method developed in constructing the universal models, we give a new direct proof that the logics axiomatized by \NNIL-axioms have the finite model property.


\end{abstract}

\keywords{Intuitionistic logic, universal model, finite model property, subframe formulas, monotonic maps.}



\section{Introduction}

\NNIL-formulas are formulas with \emph{n}o \emph{n}esting of \emph{i}mplications to the \emph{l}eft. These formulas are very expressive but considerably easier to handle and less complex than 
the class of all formulas in the language of the intuitionistic propositional calculus {\sf{IPC}}, as seen e.g., in the connection with infon logics \cite{ContGur13} and in the fact the class of \NNIL-formulas is \textit{locally finite}, i.e., there are only finitely many non-equivalent \NNIL-formulas in $n$ variables for every $n$. \todof{here \textbf{IS} the definition, perhaps invisible though...}
The study of these formulas was introduced by Visser in 1983-1984 when working on $\Sigma_1$-substitutions of propositional formulas in Heyting Arithmetic~\cite{V85}, an investigation that was continued in cooperation with de Jongh~\cite{dJV93}. Research on the purely propositional properties of these formulas was undertaken in~\cite{VNNIL}. The main tool in that article was the use of \textit{subsimulations}, a type of simulation that leaves the valuation of \NNIL-formulas intact.
It is shown that \NNIL-formulas are (up to provable equivalence) exactly the ones that are  preserved under taking submodels of Kripke models. 
In~\cite{Nic06} and in~\cite{Fan08} it was remarked that this implies that \NNIL-formulas are also preserved under taking subframes.  They axiomatize so-called subframe logics. Modal subframe logics were first introduced
by Fine~\cite{Fin85}, and intermediate subframe logics
were defined by Zakharyaschev~\cite{Zak89} (see also \cite[\S 11.3]{CZ97}), who also proved the finite model property of these logics.

It may be thought of as surprising that \NNIL-formulas axiomatize subframe logics, because for example
\cite{BG07} used $[\wedge,\to]$-formulas (i.e., formulas that have $\wedge$ and $\to$ only as connectives) to axiomatize these logics and to prove their finite model property. The $[\wedge,\to]$-formulas, though forming a locally finite class as well, are very different in character from \NNIL-formulas.\todod{I added the word class} 
To obtain \NNIL-axiomatizations, in~\cite{Nic06} (see also~\cite{BdJ15}), for each finite rooted frame $\F$, a \NNIL-formula is constructed from a model $\M$ on that frame that fails on a descriptive frame $\f{G}$ iff $\F$ is a p-morphic image of a subframe of $\f{G}$, as one calls it, a \textit{refutation criterion}. Using ideas from~\cite{BdJ15} in this paper we show that  monotonic maps can be used to describe the behavior of such formulas if the maps satisfy an additional condition: 
\textit{color-consistency}. The formulas fail on a descriptive frame $\f{G}$ iff the unraveling of $\M$ to a tree can be mapped into $\f{G}$ by a color-consistent monotonic function. This is shown to lead to the conclusion that the frame classes of intermediate subframe logics, because they are axiomatized by \NNIL-formulas, are closed under subframes in general, even of they do not satisfy the topo-subframe condition~(\Cref{cor:substructures}). This means that the class of intermediate subframe logics is at least in some ways essentially less complicated than the class of modal subframe logics. This is certainly connected to the fact shown in~\cite{VNNIL} that \NNIL-formulas are the ones of which the standard correspondent as a first order formula is equivalent to a universal one, even if it may not yet be quite clear how.\todod{Added a few sentences about the interest of \NNIL-formulas. Too wild?}  

We will further exploit the obtained refutation criterion via  color-consis\-tent monotonic maps in the present paper in two ways.
A first central result is a full description of the $n$-universal model $\Tn$ for \NNIL-formulas. This will complete the work started in~\cite{Fan08}. In fact it turns out that $\Tn$ is also the universal model for \textit{monotonically reflective} formulas ({\sf MR}-formulas), the class of formulas whose validity is reflected (or backwards preserved) by monotonic maps. \NNIL-formulas are easily seen to be \mon-formulas, and it follows essentially from~\cite{VNNIL} that  \NNIL-formulas are also exactly the ones reflected by monotonic maps. We will give an alternative proof of this result as a corollary of our universal model construction. 
The facts that logics axiomatized by \NNIL-formulas have the finite model property and are canonical will also be a simple consequence of our investigations into color-consistent monotonic maps.
As logics axiomatized by \NNIL-formulas correspond to subframe logics, these results are not new. What is new is, as stated above, that the frames of intermediate subframe logics defined by \NNIL-formulas, hence all intermediate subframe logics, are closed under arbitrary substructures not necessarily satisfying the topo-subframe condition.\todod{Deleted part. Enough?}

The paper is organized as follows: Section 2 contains the preliminaries of the paper. In Section 3 we prove the refutation criterion for \NNIL-formulas via (color-consistent) monotonic maps. Section 4 constructs universal models for \NNIL-formulas, and Section 5 proves the finite model property for logics axiomatized by \NNIL-formulas. We conclude in Section 6 by mentioning some  open problems. 

This article is largely based on~\cite{Jul19}. The proof of the finite model property (Theorem~\ref{fmp}) has already appeared in~\cite{IdeJY17}, a paper dedicated to Albert Visser.




\section{Preliminaries}


In this section, we recall briefly the most relevant terminologies and notations of this paper; for a more extensive treatment we refer the reader to \cite{BdJ15} and \cite{Nic06,CZ97,VanD}.
We fix a set {\sf Prop} of propositional variables $p$. Formulas of intuitionistic propositional logic are defined by the grammar:
\[\varphi:=\bot\mid p\mid \varphi\wedge\varphi\mid\varphi\vee\varphi\mid\varphi\to\varphi\]
As usual, we write $\neg\varphi$ for $\varphi\to\bot$. We also adopt the usual connective precedence that $\neg$ has higher precedence than $\wedge$ and $\vee$, which have precedence than $\to$. For instance $p\vee q\to \neg r\wedge s$ is read as $(p\vee q)\to ((\neg r)\wedge s)$. \todof{Added, because I got confused sometimes...} We consider the usual {\em intuitionistic propositional calculus} {\sf IPC}, and write $\vdash_{{\sf IPC}}\varphi$ or simply $\vdash\varphi$ if $\varphi$ is a theorem of {\sf IPC}. An {\em intermediate logic} $L$ is a set of formulas containing {\sf IPC} closed under modus ponens and substitution.


We have the usual Kripke semantics with intuitionistic (Kripke) frames $\f{F}=(W,R)$ and  models $\f{M}=(\f{F},V)$, where $W$ is a nonempty set of {\em worlds} (also called {\em points} or {\em nodes}), $R$ is a partial order on $W$, and  $V$ 
is a {\em persistent valuation} (i.e., $w\in V(p)$ and $wRu$ imply $u\in V(p)$). For any world $w$, define $R(w)=\{u\in W\mid wRu\}$ and $R^{-1}(w)=\{u\in W\mid uRw\}$. A set $U\subseteq W$ is said to be {\em upward closed} (or an {\em upset}) if $w\in U$ and $wRu$ imply $u\in U$. Denote by $Up(W)$ the set of all upsets of $W$. An intuitionistic general frame $\f{F}=(W,R,\Pcal)$ is an intuitionistic Kripke frame $(W,R)$ equipped with a set $\Pcal$ of upsets containing $\emptyset$ and $W$, and  closed under $\cup,\cap$ and the binary operator $\to$ on $\wp(W)$ defined as
\[U_0\to U_1:=\{w\in W\mid \forall v: wRv\text{ and }v\in U_0\text{ imply }v\in U_1\}.\]
An intuitionistic  general frame $\F$ is called {\em refined} if $\neg wRv$ implies that there is $U\in\Pcal$ such that $w\in U$ and $v\notin U$; $\F$ is called {\em compact} if for every $\mathcal{X}\subseteq \mathcal{P}\cup \{W\setminus U\mid U\in \Pcal\}$, if $\mathcal{X}$ has the {\em finite intersection property} (i.e., $\bigcap\mathcal{X}_0\neq\emptyset$ for every finite subset $\mathcal{X}_0\subseteq\mathcal{X}$), then $\bigcap\mathcal{X}\neq\emptyset$. An intuitionistic descriptive frame $\f{F}=(W,R,\Pcal)$ is an intuitionistic general frame that is both refined and compact. 
In particular, a Kripke frame $\F$ can be seen as a general frame $(\F,\Pcal)$ with $\Pcal=Up(W)$, which is clearly refined. A finite Kripke frame is also compact and thus descriptive. \todof{I changed it this way}\todod{Agreed} A descriptive model $\f{M}=(\F,V)$ is a descriptive frame $\F$ with a persistent valuation satisfying in addition $V(p) \in \mathcal P$ for all $p\in {\sf Prop}$.
Formulas in the language of {\sf IPC} are evaluated at a world $w$ in a (Kripke or descriptive) model $\mathfrak M$ recursively as follows: 
\begin{alignat*}{2}
& \mathfrak M, w \not\models \bot  \quad &&\text{always;}\\
& \mathfrak M, w \models p  \quad &&\text{iff} \quad w \in V(p);\\
&  \mathfrak M,  w \models \varphi \wedge \psi  \quad  &&\text{iff}   \quad \mathfrak M,  w \models \varphi \text{ and }  \mathfrak M, w \models \psi;  \\
& \mathfrak M, w\models \varphi \vee \psi  \quad &&\text{iff}  \quad  \mathfrak M, w \models \varphi \text{ or } \mathfrak M, w \models \psi;   \\
& \mathfrak M, w \models \varphi \rightarrow \psi  \quad  &&\text{iff}  \quad  \text{for every } u  \text{ with } w R u,~\mathfrak M, u \models \varphi \text{ implies } \mathfrak M, u \models \psi.  
\end{alignat*}
Define $V(\varphi)=\{w\in W\mid \mathfrak M, w\models\varphi\}$, and write $\mathfrak M \models \varphi$ if $V(\varphi)=W$. For a (descriptive) frame $\mathfrak F$ we write $\mathfrak F \models \varphi$ if $\mathfrak M \models \varphi$ for every model $\mathfrak M$ on $\mathfrak F$. 

In this paper, we often consider {\sf IPC} with respect to $n$ fixed propositional variables $p_1,\dots, p_n$. By an $n$-formula we mean a formula whose propositional variables are among the $n$ fixed ones only. Similarly, an $n$-model is a model $(\F,V)$ with the valuation $V$ restricted to the set consisting of the $n$ fixed propositional variables. We define the {\em color}  of a point $w$ in an $n$-model, denoted $col(w)$, as the sequence $i_1\dots i_n$ with $i_k=1$ if $p_i$ is true in $w$, and $0$ otherwise.
We write $i_1\dots i_n\leq j_1\dots j_n\ $ iff  $\ i_k\leq j_k$ for each $k=1,\dots, n$, and $i_1\dots i_n < j_1\dots j_n$ if $i_1\dots i_n \leq j_1\dots
j_n$ and $i_1\dots i_n \neq j_1\dots j_n$.

A  (Kripke or descriptive) frame $\mathfrak F'= (W', R',\mathcal{P}')$  is said to be a \emph{subframe} of a (Kripke or descriptive) frame $\mathfrak F=(W, R,\mathcal{P})$ iff $W' \subseteq W$, $R' = R\,{\upharpoonright}\, W'$, where in the case of a descriptive frame an additional {\em topo-subframe} condition needs to be satisfied, namely $W'\setminus U\in\Pcal'$ implies $W\setminus R^{-1}(U)\in\Pcal$ for all $U\subseteq W'$ (see e.g., \cite{BdJ15} for detailed discussion). \todof{\cite{BdJ15} actually refers to two other papers} We will also study subframes of descriptive frames which need not satisfy the topo-subframe condition. We call these {\em substructures}. Similarly, a model $\mathfrak M'=(\F', V' )$ is a \emph{submodel} of $\mathfrak M=(\F, V)$ iff $\F'$ is a subframe of $\F$ and $V'(p) = V(p) \cap W$ for each  $p$.
We write $\f{M}_w$
for the {\it submodel of $\f{M}$ generated by $w$}, that is, the submodel with $R(w)$ as the domain; similarly for generated subframes $\F_w$. It is easy to see that
\(\mathfrak M_w \models \varphi \text{ iff }   \mathfrak M, w \models \varphi \) for any formula $\varphi$.
If $\f{M}=\f{M}_r$ for some node $r$ (called the {\em root}), then $\M$ is said to be {\em rooted}; similarly for rooted frames. Most frames and models we consider in this paper will be rooted. 


If $wRu$, then we call $u$ a {\em successor} of $w$. If in addition $w\neq u$, then $u$ is called a {\em proper successor} of $w$, in symbols $wR^+u$. If $wR^+u$ and, for no $v$, $wR^+vRu$, $u$ is called an immediate successor of $w$. Points $w$ with no proper successor are called \emph{maximal}, that is, $wRu$ implies $w=u$.
 The {\em depth} of a point $w$ in a finite  model $\mo$ is defined as the maximal length $k$ of a chain $w=w_1R^+\dots R^+w_k$ in $\mo$, in particular, the depth of a maximal point is $1$. The \emph{depth of a finite model} is the maximal depth of the points in the model. 

A model  $(W,R,V)$ is called {\em tree-like} (or simply a {\em tree}) if for all $w\in W$, $R^{-1}(w)$ is  finite and linearly ordered; similarly for tree-like frames.
Recall that the \emph{standard unraveling}  of a rooted (infinite) model $\M=(W,R,V)$ with root $r$ is a tree-like model $\M_{\textsf{t}} =(W_{\textsf{t}}, R_{\textsf{t}}, V_{\textsf{t}})$ defined as 
\begin{itemize}
\item $W_{\textsf{t}}=\{\langle r, w_1,\dots,w_k\rangle\mid rRw_1R\dots Rw_k\}$,
\item $\sigma R_{\textsf{t}}\tau$ iff $\sigma$ is an initial segment of $\tau$,
\item $\langle r, w_1,\dots,w_k\rangle\,{\in}\, V_{\textsf{t}}(p)$ iff $w_k\in V(p)$.
\end{itemize}
For finite rooted models $\M$, the unraveling $\M_{\textsf{t}}$ is defined the same way except that $W_{\textsf{t}}$ is defined using immediate successorship instead of successorship $R$.
We identify the root $\langle r\rangle$ of the unraveled model $\M_{\textsf{t}}$ with the root $r$ of the original model $\M$, and write simply $r$ for $\langle r\rangle$. \todof{deleted the notation $\mathfrak{T}_{\M}$}



One central notion of our revisit to \NNIL-formulas is the notion of monotonic map. A \emph{monotonic map} between two frames $\f{G}=(W,R)$ and $\F=(W',R')$ is a function $f:W\to W'$  that preserves order, i.e., \todof{I swapped $\f{G}$ and $\F$ here, because $\F$ is usually paired with $\M$ and $\f{G}$ with $\N$, I think.}
\begin{description}
\item[(order preservation)] $wRu$ implies $f(w)R'f(u)$.
\end{description}
A {\em color-preserving monotonic map} $f:\N\to \M$, 
denoted as $\M\leq_f \N$, is a monotonic map between the two underlying frames that also preserves colors, i.e., \todof{changed to this, in the end}\todod{left 'between two models' out once, it made me doubt the direction} 
\begin{description}
\item[(color preservation)] $col(f(w))=col(w)$.
\end{description}
We write $\M\leq \N$ if there exists some color-preserving monotonic map $f$ such that $\M\leq_f \N$.
%
Note that  {\em functional subsimulations}, which played an important role in the previous study \cite{VNNIL} of \NNIL-formulas, are, in fact, color-preserving monotonic maps.  
Let us also recall that the familiar p-morphisms are 
color-preserving monotonic maps  $f$ between two models that also satisfy: 
\begin{description}
\item[(back condition)] $f(w)R'u'$ implies $\exists u\in W'(wRu$ and $f(u)=u')$.
\end{description}
As an example, the \emph{natural map} $\alpha$ between an unraveled model $\M_{\textsf{t}}$ and the original model $\M$, defined as $\alpha(\langle r,w_1,\dots,w_k\rangle)=w_k$, is a surjective p-morphism.  It is easy to see that the truth of a formula $\varphi$ is invariant under p-morphisms $f$ between two models $\N$ and $\M$, i.e.,
\[\M,f(w)\models\varphi\iff \N,w\models\varphi.\]

Let us now recall from \cite{VNNIL,Fan08} some basic facts about \NNIL-formulas, which are formulas in the language of {\sf IPC} with no nesting of implications to the left. For example,  $p\vee q\to (r\to s)$ and $p\to (q\to\bot)\vee(r\to s)$ are \NNIL-formulas, whereas $(p\to\bot)\to \bot$ and $(p\to q)\vee r\to p$ are not.   \NNIL-formulas form a  {\em locally finite} class of formulas, that is, for every natural number $n$, there are only finitely many non-equivalent \NNIL-formulas in $n$ propositional variables. For short, we may say somewhat improperly that \NNIL-formulas are locally finite. \todod{Alright?}\todof{I thought we have treated already this "locally finite" issue with the definition a sentence before. Adding one more sentence looks like an exaggeration...}\todod{The last sentence says something different} 
Since conjunctions and disjunctions in front of implication can be eliminated using the equivalences \todof{treated ``locally finite" and NF}
\[\vdash(\varphi\wedge\psi\to\chi)\leftrightarrow (\varphi\to(\psi\to\chi)\,\text{and}\vdash(\varphi\vee\psi\to\chi)\leftrightarrow ((\varphi\to\psi)\wedge(\varphi\to\chi)),\]
every \NNIL-formula can be brought into an equivalent \NNIL-formula in the following normal form: 

\begin{definition}
{\sf NNIL}-formulas in {\em normal form} are defined by the grammar:
\[\varphi:= \bot\mid p\mid\varphi\wedge\varphi\mid\varphi\vee\varphi\mid p\to\varphi\]
\end{definition}


The  approach of this paper is based on the observation made in \cite{BdJ15} that \NNIL-formulas are  reflected (or backwards preserved) by color-preserving monotonic maps. We recall this fact  in detail as follows.


\begin{lemma}\label{nnil}\cite{BdJ15}
Let $\varphi$ be a \NNIL-formula. For any two models $\N=(W,R,V)$ and $\mathfrak M=(W',R',V')$  such that $\M\leq_f \N$ for some color-preserving monotonic map $f:\N \to \M$, we have that for any $w \in W$,
\begin{equation}\label{eq:1}
\M, f(w)\models\varphi \Longrightarrow \N, w\models\varphi.
\end{equation}
In particular, if $\M\leq\N$ and $\M\models\varphi$, then $\N\models\varphi$.
\end{lemma}
\begin{proof}
The proof is a routine induction on $\varphi$. Assume $\varphi$ to be in normal form. Only the case $\varphi=p\to \psi$ is non-trivial. Suppose $\M, f(w) \models p \to\psi$ and $\N, u \models p$ for some $u$ with $wRu$. Since $f$ is monotonic and color-preserving, $f(w)R' f(u)$ and $\M, f(u) \models p$, thus $\M, f(u) \models\psi$. By the induction hypothesis, we obtain $\N, u\models\psi$, as required.
\end{proof}

The above lemma also gives rise to the class \MR (short for {\em monotonically reflective}) of formulas that are  
reflected by color-preserving monotonic maps.\footnote{The class {\sf MR} was called {\sf SR} (short for {\em stably reflective}) in~\cite{Jul19}. } Obviously we have $\NNIL\subseteq\MR$.

The identity map from a submodel $\f{N}$ of $\M$ to $\M$ itself is obviously a color-preserving monotonic map. 
Consequently, \NNIL-formulas $\varphi$ are preserved under submodels, that is, $\M\models\varphi$ implies $\f{N}\models\varphi$. It was shown in \cite{VNNIL} that the converse holds as well, namely, every formula  preserved under taking submodels is (equivalent to) a \NNIL-formula. From this it also follows that every \MR-formula (being preserved under taking submodels) is equivalent to some \NNIL-formula, and thus $\NNIL=\MR$. In this paper (in \Cref{cor:betaaxiomatization}) we provide a direct alternative proof of these results  by means of monotonic maps. 

\section{\NNIL-subframe formulas and monotonic maps on trees}


In this section we present a refutation criterion for \NNIL-subframe formulas via monotonic maps. 
\NNIL-subframe formulas were first introduced in~\cite[\S3.3]{Nic06} as formulas axiomatizing subframe logics in \NNIL-form. They were inspired by the Jankov-de Jongh formulas, in fact they were introduced together in \cite{Nic06}. In the universal model of {\sf IPC} (see Definition~\ref{def:unipc}) the Jankov-de Jongh formulas characterize point generated upsets. Their validity can thus be translated into a tangible semantic condition which leads to a refutation criterion (known as the Jankov-de Jongh Theorem). In the case of the \NNIL-subframe formulas the refutation condition (via p-morphisms) comes immediately and was used in \cite{Nic06} to show that these formulas axiomatize all {\em subframe logics} (i.e., logics whose class of frames is closed under subframes). 

In \cite{Nic06} (see also \cite{BdJ15}) the \NNIL-subframe formulas  were introduced as certain \NNIL-formulas $\beta(\mathfrak{F})$ constructed inductively out of arbitrary finite rooted frames $\F$. Such constructions make  sense  for arbitrary finite models as well. We  now define  \NNIL-subframe formulas $\beta(\N)$ (in $n$ variables) with respect to arbitrary finite $n$-models $\N$ in the same manner. This slight difference in the definition will enable us to prove a simpler refutation criterion for  \NNIL-subframe formulas via monotonic maps, which will be important for the remaining sections of the paper. 

\begin{definition}\label{definition subframe formula}
Let $\N=(W,R,V)$ be a finite  $n$-model. For every $w\in W$, we define a \NNIL-formula $\beta(w)$ by induction on the depth of $w$ as follows:
\begin{itemize}
\item If $w$ is a maximal point of
$\f{N}$,  define
$$
\beta(w):=\bigwedge prop(w)\to \bigvee notprop(w),
$$
where
\[prop(w):=\{p_i\mid \N,w\models p_i,~1\leq i\leq n\}\]
\[\text{and }notprop(w):=\{p_i\mid \N,w\not\models p_i,~1\leq i\leq n\}.\] 
\item If $w$ is not maximal, and $w_1,\dots, w_k$ are all of its immediate successors with
$\beta(w_{i})$ already defined for every $w_i$, then define
$$
\beta(w):=\bigwedge prop(w)  \to
\bigvee notprop(w)\vee \bigvee_{i=1}^k \beta(w_i).
$$
\end{itemize}
If $\mathfrak N$ is rooted with root $r$, we define $\beta(\f{N})=\beta(r)$. 
\end{definition}

Note that the formula $\beta(w)$ is also parameterized with the natural number $n$, which is given by the $n$-model $\N$ that the node $w$ is taken from. \todof{added this sentence for the two referees.}



\begin{lemma} 
\label{lem:selfrefutetation}
For any finite $n$-model $\N$, we have $\N, w \not \models \beta(w)$.
\end{lemma}

\begin{proof} 
We prove the lemma by induction on $d(w)$. If $d(w)=1$, clearly, 
\begin{equation}\label{lem:selfrefutetation_eq1}
\N, w\models\bigwedge prop(w)~\text{ and  }~\N, w\not\models\bigvee notprop(w), 
\end{equation}
which give $\N,w\not\models \beta(w)$. 

Suppose $d(w)>1$ and the lemma holds for all nodes with depth less than $d(w)$. Assume that $w_1,\dots,w_k$ are immediate successors of $w$. By induction hypothesis, we have $\N,w_i\not\models\beta(w_i)$ for all $1\leq i\leq k$. Thus, we obtain $\N,w\not\models\bigvee_{i=1}^k\beta(w_i)$ by persistency. Since (\ref{lem:selfrefutetation_eq1}) also holds for $w$, we conclude $\N,w\not\models\beta(w)$.
\end{proof}

We now  prove our new refutation criterion for the \NNIL-subframe formulas $\beta(\N)$ via monotonic maps. 
In this criterion and also in other discussions in the sequel, we will consider the unraveled tree-like models $\N_{\textsf{t}}$ instead of the finite rooted $n$-models $\N$ themselves. For reasons that will become apparent in the detailed proofs, it is in fact technically crucial to do so. This subtlety  was already apparent in the previous study of universal models for {\sf{NNIL}}-formulas in~\cite{Fan08}. 
Since any node $w$ in the finite $n$-model $\N$ with root $r$ and the corresponding node $\langle r,\dots,w\rangle$ in $\N_{\textsf{t}}$ have the same color and essentially the same set of immediate successors, one can show by induction that the two formulas $\beta(\N)$ and $\beta(\mathfrak N_{\textsf{t}})$ are actually identical. We will thus not distinguish between the two formulas $\beta(\N)$ and $\beta(\mathfrak N_{\textsf{t}})$.


\begin{theorem}\label{betamono1}\label{betamonothm}
Let $\M$ be an $n$-model and $\mathfrak N$ a finite rooted  $n$-model. Then, $\M \not \models \beta(\mathfrak N)$ iff $\M\leq \N_{\textsf{t}}$.

In particular, for any (Kripke or descriptive) frame $\F$, we have that $\F\not\models\beta(\mathfrak N)$ iff $\M\leq \f{N}_{\textsf{t}}$ for some model $\M$ on $\F$.

\end{theorem}
\begin{proof}
Suppose first that $\M\leq \N_{\textsf{t}}$ and $r$ is the root of $\N_{\textsf{t}}$. By Lemma \ref{lem:selfrefutetation}, we have $\N_{\textsf{t}}, r \not\models\beta(r)$. Since $\beta(r)\in\NNIL$, we obtain by \Cref{nnil} that $\M\not\models\beta(r)$, i.e., $\M\not\models\beta(\N)$. \todof{Ok, as you suggested I simplified this paragraph a little bit, but only a little bit gets simplified actually...}


Conversely, assuming $\M \not \models \beta(\mathfrak N)$ we define a color-preserving monotonic map $f:\N_{\textsf{t}}\to\M$ by defining the value $f(u)$
by induction on the depth of $u$ in $\N_{\textsf{t}}$. We will guarantee that $f$ is monotonic and color-preserving by guaranteeing the following: (i) $f(u)$ has the color of $u$. (ii) For the unique immediate predecessor $w$ of $u$, $f(u)$ is a successor of $f(w)$. (iii) For every immediate successor $u_i$ of $u$,  $\M,f(u)\not\models\beta(u_i)$.


We first define $f(r)$.
Since ${\mathfrak M}\not\models \beta(r)$, there exists a node $x$ in $\M$ such that  \todof{rewrote the whole proof}
\begin{equation}\label{betamono1_eq1}
{\mathfrak M},x\models \bigwedge prop(r),~ {\mathfrak M},x\not\models \bigvee notprop(r)\text{ and }{\mathfrak M},x\not\models \beta(r_i)
\end{equation}
for each immediate successor $r_i$ of $r$ (if any). We define $f(r)=x$. Clearly (\ref{betamono1_eq1}) implies that $col(f(r))=col(x)=col(r)$ and $\M,f(r)\not\models\beta(r_i)$.

Next, we define $f(u)$ for $u$ properly above $r$. Let $w$ the unique immediate predecessor of $u$ in the unraveled tree-like model $\N_{\textsf{t}}$, and suppose that $f(w)$ has already been defined. Similar to the above, since we have guaranteed that $\M, f(w) \not \models \beta(u)$, there is a successor $x$ of $f(w)$ for which the corresponding clause (\ref{betamono1_eq1}) holds for $x$ and $u$. Define $f(u) =x$. By definition $f(u)$ is a successor of $f(w)$. Again, (\ref{betamono1_eq1}) implies that $col(f(u))=col(x)=col(u)$ and $\M,f(u)\not\models\beta(u_i)$ for each immediate successor $u_i$ of $u$.
%
\end{proof}


Let $\mathsf{B}$ denote the collection of all \NNIL-subframe formulas of finite models as defined in \Cref{definition subframe formula}, i.e. 
\begin{equation*}
\mathsf{B}= \{ \beta(w) \mid w \text{ is a node in some finite $n$-model for some $n$}\}.
\end{equation*}
Obviously, $\mathsf{B}$  is included in the class of \NNIL-formulas. As mentioned already, N. Bezhanishvili \cite{Nic06} gave the refutation criterion for formulas in $\mathsf{B}$ via p-morphisms, which are color-preserving monotonic maps with extra conditions. In this sense our \Cref{betamonothm} is simpler than the corresponding one in \cite{Nic06}. On the basis of the refutation criterion N. Bezhanishvili was able to prove that the formulas in $\mathsf{B}$ are sufficient to axiomatize the intermediate subframe logics and hence

\begin{theorem}[\cite{Nic06}, Cor. 3.4.16]\label{subframe_log_thm}
All intermediate subframe logics are axiomatized by \NNIL-formulas.
\end{theorem}

Recall that by \Cref{nnil} all \NNIL-formulas are in $\mon$. As another consequence of \Cref{betamono1}, the three formula classes $\textsf{B}$, $\NNIL$ and $\mon$ distinguish the same finite pointed models in the sense of the following definition. This result for $\NNIL$- and $\mon$-formulas follows also already from \cite{VNNIL}, by a different argument. 


\begin{definition}\label{phi_indist_df}
Let $\Phi$ be a class of formulas. Two pairs $(\M, w)$ and $(\f{N}, u)$ of models with nodes in the corresponding domains are said to be \emph{$\Phi$-equivalent}, written $(\M, w)\simeq_{\Phi}(\f{N}, u)$, if for each $\varphi \in \Phi$,
\begin{equation*}
\M, w \models \varphi  \iff  \mathfrak N, u \models \varphi.
\end{equation*}
Similarly, we write $\M\simeq_{\Phi}\f{N}$ if for each $\varphi \in \Phi$,
\begin{equation*}
\M \models \varphi  \iff  \mathfrak N \models \varphi.
\end{equation*}
 
\end{definition}
\begin{proposition}
\label{prop:equivalent}
Let $\M$ and $\mathfrak N$ be finite models with nodes $w$ and $u$ in the corresponding domains, respectively. The following are equivalent:
\begin{enumerate}[label=(\roman*)]
\item $(\M, w)\simeq_{\mathsf{B}}(\mathfrak N, u)$.
\item $(\M, w)\simeq_{\NNIL}(\mathfrak N, u)$.
\item $(\M, w)\simeq_{\mon}(\mathfrak N, u)$.
\end{enumerate}
\end{proposition}
\begin{proof}
The implications (iii) $\Rightarrow$ (ii) $\Rightarrow$ (i) are obvious since $\mathsf{B} \subseteq \mathsf{NNIL} \subseteq \mon$. 
We show that (i) implies (iii). Assume that $(\M, w)\simeq_{\mathsf{B}}(\mathfrak N, u)$. 
For any $\varphi \in \mon$, we only show the direction that  $\M, w \models \varphi$ implies $\mathfrak N, u\models \varphi$; the other direction is symmetric. Assume that $\varphi$ is an $n$-formula and view $\N$ and $\M$ as $n$-models by restricting the valuations to the variables of $\varphi$. 
By Lemma \ref{lem:selfrefutetation}, $\mathfrak N, u \not \models \beta(u)$, which by $\mathsf{B}$-equivalence implies $\M, w \not \models \beta(u)$, or $\M_w, w \not \models \beta(u)$. Now, by  \Cref{betamono1}, there is a color-preserving monotonic map $f$ from the unraveling $(\mathfrak N_{u})_{\textsf{t}}$ into $\M_w$. Clearly $f(u)$ is a successor of $w$. 
Thus, by assumption and persistency, we have $\M_w, f(u) \models \varphi$, which implies  $(\mathfrak N_{u})_{\textsf{t}}, u \models \varphi$ as $\varphi \in \mon$. Hence $\mathfrak N, u \models \varphi$, as required. 
\end{proof}

Next, we generalize  \Cref{betamonothm} by relaxing it from a statement about a color-preserving monotonic map into a model on a frame $\F$ to a statement about a color-consistent monotonic map into the frame $\F$.  We call a monotonic map $f$ from an $n$-model $\N$ into a frame $\F=(W,R)$  \textit{color-consistent} if  for all points $w,u$ in $\N$, 
\[f(w)Rf(u)\Longrightarrow col(w)\leq col(u).\]






\begin{theorem}\label{betamonothm-frame}
Let $\F$ be a (Kripke or descriptive) frame, and $\mathfrak N$ a finite rooted $n$-model. 
Then, $\F\not\models\beta(\mathfrak N)$ iff there is a color-consistent monotonic  map from $\N_{\textsf{t}}$ into $\F$. 
\end{theorem}
\begin{proof} The left to right direction follows from \Cref{betamonothm}, as a color-preserving\todod{corrected this from consistent} monotonic map into a model is clearly color-consistent. \todof{$\leftarrow$Now I see calling it a monotonic map (between models) is indeed confusing, especially at this place}\todof{It is fine now.} 
For the other direction assume that $f$ is a monotonic color-consistent map from $\N_{\textsf{t}}$ into $\F=(W,R,\mathcal{P})$. 

\begin{claim} 
For every $w$ in $\N_{\textsf{t}}$, there exists $U_w\in \Pcal$ such that $f(w)\in U_w$, and for every other node $u$ in $\N_{\textsf{t}}$, $f(u)\in U_w$ iff $f(w)Rf(u)$.
\end{claim}
\begin{proofclaim}
Define $U_w=\bigcap\{V_{u}\mid \neg f(w)Rf(u)\}$, where $V_{u}$ is an upset in $\Pcal$ containing $f(w)$ but not $f(u)$, whose existence is guaranteed by $\Pcal$ being refined (which is the case also when $\F$ is a Kripke frame). \todof{new sentence, see also p4 the definition of descriptive frame}
\end{proofclaim}

Now, define $V(p) =\bigcup\{U_w\mid \N_{\textsf{t}},w \models p\}$, where each $U_w$ is as in the claim. 
We verify that $f$ preserves colors between $\N_{\textsf{t}}$ and the model $(\F,V)$, 
and
$\F \not \models \beta(\mathfrak N)$ will then follow by \Cref{betamonothm}. Now, by the definition of $V$ and the claim, $\N_{\textsf{t}},w\models p$ implies $f(w)\in U_w\subseteq V(p)$, and thus $f(w)\in V(p)$. Conversely, if $f(u)\in V(p)$, then there exists $w$  such that $f(u)\in U_w$ and $\N_{\textsf{t}},w\models p$. The former implies by the claim that $f(w)Rf(u)$. Since $f$ is color-consistent, we have $col(w)\leq col(u)$, which implies that $\N_{\textsf{t}},u\models p$, as required.
\end{proof}

We end this section by deriving as an immediate consequence of the above theorem the surprising and important consequence that $\beta(\mathfrak N)$-formulas are preserved by arbitrary substructures. We say that a (Kripke or descriptive) frame $\mathfrak G=(W', R',\Pcal')$ is a \emph{substructure} of another frame $\mathfrak F= (W, R,\Pcal)$ iff $W' \subseteq W$ and $R'=R\,{\upharpoonright}\, W'$. 
\begin{corollary} 
\label{cor:substructures}
Let $\mathfrak F$ and $\mathfrak G$ be (Kripke or descriptive) frames with $\mathfrak G$ a substructure of $\mathfrak F$. If $\mathfrak F \models \beta(\mathfrak N)$ for some finite $n$-model $\mathfrak N$, then $\mathfrak G \models \beta(\mathfrak N)$. In other words, validity of formulas in $\mathsf{B}$ is preserved by substructures. 
\end{corollary}
\begin{proof}
Suppose for contraposition that $\mathfrak G \not \models \beta(\mathfrak N)$. By Theorem \ref{betamonothm-frame}, there is a color-consistent monotonic map $f: \N_{\textsf{t}} \rightarrow \mathfrak G$. The map $f$ composed with the embedding   from $\mathfrak G$ into $\mathfrak F$ 
 is easily seen to be color-consistent. Thus, we conclude $\mathfrak F \not \models \beta(\mathfrak N)$ by applying \Cref{betamonothm-frame} again.   
\end{proof}

Not even in the descriptive case substructures impose any relation between the admissible sets of $\mathfrak G$ and $\mathfrak F$. This contrasts with the definition of the topo-subframes that are needed in modal logic. In modal logic the corresponding result does not apply, as a subframe logic like {\sf GL} is an obvious counterexample.  
This property, together with \Cref{subframe_log_thm}, immediately implies that intermediate subframe logics are canonical. \todof{added  \Cref{subframe_log_thm} as pointed out by referee 2} We will formally state this result and also generalize it to logics axiomatized by arbitrary \NNIL- or $\mon$-formulas  in the next section in \Cref{BmonNNIL-logics,cor:canonical}. 

\section{Universal models}\label{sec:uni}

In this section we construct $n$-universal models for \NNIL-formulas. This is a continuation of the project started in~\cite{Fan08} in which among other things the 2-universal model for \NNIL was constructed using the subsimulations of~\cite{VNNIL}. In this present paper 
we will construct the universal models using color-preserving monotonic maps instead. By the results of the previous section these models will actually also be universal models for $\mon$-formulas. We will then be able to derive that \NNIL-formulas are exactly the ones that are reflected 
by color-preserving monotonic maps, an important result occurring indirectly 
already in the earlier study of \NNIL-formulas \cite{VNNIL}. We will also formally conclude, as a consequence of \Cref{cor:substructures} in the previous section, that logics axiomatized by \NNIL- or $\mon$-formulas are canonical. 

\subsection{Definition and properties of universal models}

 Universal models for modal logics and {\sf IPC} were thoroughly investigated by a number of authors~\cite{B86,deJYang,Grigolia87,Rybakov97,Shehtman78} (see \cite[\S8]{CZ97} and \cite[\S3]{Nic06} for an overview), and results for fragments of {\sf IPC} can be found in~\cite{gool,tzimoulis,Lex96,renardel}. \todof{I moved this sentence to here, because I think it does not belong to where it were (after Def. 4.1). This was basically mentioned by reviewer 2.} We now give a formal and general definition for the notion of $n$-universal model  that applies to
rather different fragments of {\sf IPC} in a uniform manner. Such a general definition will also allow us  to point out the place and usefulness of universal models. \todod{changed and added}\todof{Now made a somewhat radical change here to the text}


\begin{definition} \label{def:universalmodel} 
We say that a model $\M=(W,R,V)$ is an \emph{$n$-universal model} of a class $\Phi$ of $n$-formulas iff every upset generated by a single point in $W$ is finite, 
and\todod{ added this}\todof{rephrased a little} the following conditions are satisfied:
\begin{enumerate}[label=(\roman*)]
\item 
For any $\varphi,\psi\in \Phi$, if $\varphi\nvdash\psi$, then there exists $w\in W$ such that $w\models\varphi$ and $w\not\models\psi$.

\item For every 
upset $U$ of $\M$ generated by a single point,  
there is $\varphi \in  \Phi$ such that $V(\varphi)= U$. \todof{Why do we require that $U$ is finite?} \todod{Because usually point-generated upsets in universal models are finite. I suggest we leave it out as I did. Takes care of a remark by Referee 2}. \todod{Ultimately I have been more radical. Read below.}
\end{enumerate}
\end{definition}

Condition (i) in the $n$-universal model makes sure that ``there are enough counterexamples", while condition (ii) ensures that ``there are no superfluous" points. By taking the contrapositive of (i) we  know  that for any two formulas $\varphi,\psi\in \Phi$, if $V(\varphi)\subseteq V(\psi)$ in the universal model, then $\vdash\varphi\to\psi$; in particular,  $V(\varphi)=V(\psi)$ implies $\vdash\varphi\leftrightarrow\psi$. Condition (ii) implies that no two distinct worlds in the $n$-universal model of $\Phi$ are $\Phi$-equivalent.  For formula classes $\Phi$ that have the finite model property (e.g., for {\sf IPC}), a model satisfying conditions (i) and (ii) automatically satisfies the requirement that point-generated upsets are finite. In general only for a locally finite formula class $\Phi$ do universal models seem to be useful.  
 \todod{Added this.}\todof{rephrased a little}\todod{And I changed it once more: to locally finite}


 The definable upsets in $n$-universal models will
reflect the algebraic structure of the logic or fragment.\todod{ changed this sentence} In the case of {\sf IPC} itself it will be the Lindenbaum-Tarski algebra, a Heyting Algebra, but this is also the case for the fragment of the formulas with only $\to$ and $\wedge$, because even without $\vee$ it forms a distributive lattice, which in the finite case always is a Heyting algebra. For the fragment of {\sf IPC} with only $\to$ this is no longer the case but it still gives suitable $n$-universal models~\cite{renardel}. For the set of all \NNIL-formulas in an arbitrary number of variables the Lindenbaum-Tarski algebra is not a Heyting algebra since, although they do form a distributive lattice, they are not closed under implication. But the \NNIL-formulas in $n$ variables, by their local finiteness, do automatically form a Heyting algebra.


Since these algebras are free algebras there is a close relationship with the $n$-canonical model (also known as $n$-Henkin model). Usually the $n$-universal model is the ``upper part" of the $n$-Henkin model (see~\cite{Nic06,deJYang}), and, in the case of locally finite fragments, isomorphic to it (see~\cite{gool,tzimoulis}). Unsurprisingly, it will turn out in\todod{ changed this piece}
Theorem~\ref{canonical_universal} that indeed the $n$-universal model is also in the \NNIL-case isomorphic to the $n$-canonical model.


The  mappings appropriate to the logic or the fragment play an important role in universal models. In general $n$-universal models have the property that any finite $n$-model is connected by such a map to a unique generated submodel of the universal model. This then gives a uniqueness property to the universal model and will imply that it is in a certain sense the smallest. For {\sf IPC} itself these mappings are the p-morphisms, for \NNIL their role will be taken over by the color-preserving monotonic maps.

All this means that it is often advantageous to see the $n$-universal model not as one model but as the collection of all of its point-generated submodels, ordered by the generated \todod{added }submodel relation. In fact, the latter was the only point of view in the proto-universal models of~\cite{deJ68}. These different views of the model are pertinent in the case of the $n$-universal model for \NNIL-formulas about to be constructed. In particular, every  node in our $n$-universal model will be associated with a local \todod{added} tree that is mostly very different from the submodel generated by the node: the ordering of the local models
is no longer the generated submodel relation, and thereby the isomorphism between the local model and the submodel of the universal model generated by it has been lost (although it can be recovered to a certain extent, see~\Cref{prop:properties Tn2}). An arbitrary finite $n$-tree will have a unique map to such a local tree  in the $n$-universal model only. In fact, 
as we will show in \Cref{Tn_turth_thm},
it can be monotonically mapped back and forth by color-preserving maps into that unique  local tree. 
\\



Let us first briefly recall the construction of the $n$-universal model $\U(n)$ of {\sf IPC} (see e.g.~\cite{deJYang}). 
The construction of the $n$-universal model for $\NNIL$-formulas to be given in \Cref{univnn}, though more complex, basically follows the same strategy. 

\begin{definition}\label{def:unipc}
The $n$-universal model  $\U(n)=(W,R,V)$ of {\sf IPC} is  defined inductively in layers as follows.
\begin{itemize}
\item The first layer consists of  nodes of the $2^n$ distinct $n$-colors.
\item Assume that the $\leq\!m$th layers have been defined already. We define the $(m+1)$th layer as follows:
\begin{itemize}
\item For each element $w$ in the $m$th layer, and each color $c<col(w)$, add a new node $u$ to layer $m+1$ with color $c$ and with $w$ the only immediate successor of $u$.
\item For each set $X=\{{w_1},\dots,{w_k}\}$ ($k\geq 2$)  of pairwise $R$-incomparable elements in layers $\leq m$ containing at least one member from layer $m$, and each color $c$ less than or equal to the color of all nodes in $X$, add a new node $w$ to layer $m+1$ with color $c$ and $w_1,\dots,w_k$ as immediate successors. 
\end{itemize}
\end{itemize}
\end{definition}

The $1$-universal model of \textsf{IPC} is also known as Rieger-Nishimura Ladder \cite{Nis60,Rie49}. See  \Cref{fig:u2ipc} for a fragment of the $2$-universal model $\U(2)$. 

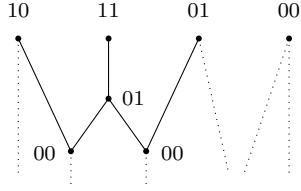
\begin{figure}[t]
\begin{center}
\begin{tikzpicture}[transform shape, scale=1]
\draw (-1.2,0) node[inner sep=0.8pt, circle, fill, label={[shift={(0,0.1)}]{\scriptsize 10}}] {} -- (-.5,-1.5) node[inner sep=0.8pt, circle, fill, label={[shift={(-0.35,-0.3)}]{\scriptsize 00}}] {};

\draw (-.5,-1.5) -- (0,-.8) node[inner sep=0.8pt, circle, fill, label=right:{\scriptsize 01}] {} -- (0,0) node[inner sep=0.8pt, circle, fill, label={[shift={(0,.1)}]{\scriptsize 11}}] {};

\draw  (0,-.8)  -- (.5,-1.5) node[inner sep=0.8pt, circle, fill, label={[shift={(0.35,-0.3)}]{\scriptsize 00}}] {} -- (1.2,0) node[inner sep=0.8pt, circle, fill, label={[shift={(0,0.1)}]{\scriptsize 01}}] {};

\draw[dotted]  (2.4,0) node[inner sep=0.8pt, circle, fill, label={[shift={(0,0.1)}]{\scriptsize 00}}] {} -- (2.4,-1.8);

\draw[dotted]  (2.4,0) -- (1.8,-1.8);

\draw[dotted]  (1.2,0)  -- (1.6,-1.8);

\draw[dotted]  (-1.2,0)  -- (-1.2,-1.8);

\draw[dotted]  (.5,-1.5) -- (.5,-2);
\draw[dotted]  (-.5,-1.5) -- (-.5,-2);



\end{tikzpicture}
\caption{A fragment of $\U(2)$}
\label{fig:u2ipc}
\end{center}
\end{figure}

We now construct the $n$-universal model $\Tn$ of \NNIL-formulas. It will turn out that this model is also the $n$-universal model of $\mon$-formulas.
The nodes in our universal model $\Tn$ will be  finite tree-like models  themselves, and we thus denote them as $T_w$, $T_u$, etc. The reader may think of $w,u$ as elements from $\U(n)$, and $T_w$ and $T_u$ as the unravelings of $\mathcal{U}(n)_w$ and $\mathcal{U}(n)_u$ to trees; in particular, the color of the root of $T_w$ is taken to be $col(w)$. 
We will  not use here other facts though concerning this ``embedding" of $\Tn$ into $\U(n)$. Each chain in the $T_w$'s will be strictly decreasing in color, but not all such nodes in $\U(n)$ will participate in $\Tn$.
\todof{Changed this piece, and moved the definition of $\leq$ forward.} The ordering in our $\Tn$ will not be the usual generated submodel ordering but the $\leq$ relation between models induced by color-preserving monotonic maps. 
We write $\M\equiv \f{N}$ if both $\M\leq \N$ and $\N\leq \M$. 
It is easy to verify that the relation $\leq$ is reflexive and transitive, and thus 
$\equiv$ is an equivalence relation. \todof{I've deleted one sentence here (see the source codes), as I don't see the point of the sentence.}




\begin{definition}\label{univnn}
The $n$-model $\Tn=(W,\leq,V)$ is defined  as follows: 
\begin{itemize}
\item The domain $W$ is defined inductively in layers:
\begin{itemize}
\item The first layer consists of nodes (or trees of single nodes) of the $2^n$ distinct $n$-colors.
\item Assume that the $\leq\!m$th layers have been defined already. We define the $(m+1)$th layer as follows:
For every set $X=\{T_{w_1},\dots,T_{w_k}\}$  of pairwise $\leq$-incomparable  trees in layers $\leq m$ containing at least one member of layer $m$, and every color $c$  strictly smaller than all the colors  at the roots of the trees in $X$,
 build a tree $T_w$ by taking the disjoint union of the trees in $X$ and adding a fresh root $w$ of color $c$ below. Then add $T_w$ as a new node to layer $m+1$.

\end{itemize}
\item Order (the trees in) $W$ by the $\leq$ relation. 
\item The color of a node $T_w$ in $W$ is defined as the color of the root $w$ in the tree $T_w$.  

\end{itemize}
\end{definition}

\afterpage{\clearpage}
\begin{figure}[!th]
\begin{center}
\begin{tikzpicture}[transform shape, scale=1]

\draw (-4.5,9) -- (-5.8,7);
\draw (-4.5,9) -- (-2.7,7);
\draw (-1.5,9) -- (-5.8,7.2);
\draw (-1.5,9) -- (5.8,7.2);
\draw (-1.5,9) -- (0,7.2);
\draw (1.5,9) -- (-2.7,7.3);
\draw (1.5,9) -- (-0.1,7.2);
\draw (1.5,9) -- (2.7,7.4);
\draw (4.5,9) -- (5.8,7);
\draw (4.5,9) -- (2.7,7);

\draw (-5.8,6.3) -- (-4.5,0.8);
\draw (5.8,6.3) -- (4.5,0.8);
\draw (-2.9,6.3) -- (0,4.5);
\draw (-2.9,6.3) -- (-3.1,4.5);
\draw (2.9,6.3) -- (0,4.5);
\draw (2.9,6.3) -- (3.1,4.5);
\draw (0,6.3) -- (-2.8,4.3);
\draw (0,6.3) -- (2.8,4.3);

\draw (-3.3,2.7) -- (0,1);
\draw (3.3,2.7) -- (0,1);
\draw (-3.3,2.7) -- (-4,1.2);
\draw (3.3,2.7) -- (4,1.2);

\draw (-4,-0.4) -- (-2,-2);
\draw (0,0) -- (-2.5,-2);
\draw (4,-0.4) -- (2,-2);
\draw (0,0) -- (2.5,-2);

\draw (-2.5,-3) -- (1,-5.2);
\draw (2.5,-3) -- (-1,-5.2);


\path[fill=gray!20] (-4.4,9.5) ..
                 controls +(-26:0.5cm) and +(5:0.3cm) 
.. (-4.4, 8.7)
..                 controls  +(180:0.8cm) and +(160:0.8cm)
.. cycle;
\draw (-4.5,9) node[inner sep=0.8pt, circle, fill, label={above:{\scriptsize 10}}] {};

\path[fill=gray!20] (-1.4,9.5) ..
                 controls +(-26:0.5cm) and +(5:0.3cm) 
.. (-1.4, 8.7)
..                 controls  +(180:0.8cm) and +(160:0.8cm)
.. cycle;
\draw (-1.5,9) node[inner sep=0.8pt, circle, fill, label={above:{\scriptsize 11}}] {};

\path[fill=gray!20] (1.6,9.5) ..
                 controls +(-26:0.5cm) and +(5:0.3cm) 
.. (1.6, 8.7)
..                 controls  +(180:0.8cm) and +(160:0.8cm)
.. cycle;
\draw (1.5,9) node[inner sep=0.8pt, circle, fill, label={above:{\scriptsize 00}}] {};

\path[fill=gray!20] (4.6,9.5) ..
                 controls +(-26:0.5cm) and +(5:0.3cm) 
.. (4.6, 8.7)
..                 controls  +(180:0.8cm) and +(160:0.8cm)
.. cycle;
\draw (4.5,9) node[inner sep=0.8pt, circle, fill, label={above:{\scriptsize 01}}] {};


\path[fill=gray!40] (-5.7,5.8) ..
                 controls +(5:0.5cm) and +(-5:1cm) 
.. (-5.7, 7.5)
..                 controls  +(180:1.2cm) and +(180:0.5cm)
.. cycle;
\draw (-5.8,6.3) node[inner sep=0.8pt, circle, fill, label={[shift={(0.2,-0.45)}]{\scriptsize 10}}] {} -- (-5.8,7) node[inner sep=0.8pt, circle, fill, label={[shift={(0.2,-0.05)}]{\scriptsize 11}}] {};

\path[fill=gray!40] (-2.8,5.8) ..
                 controls +(5:0.5cm) and +(-5:1cm) 
.. (-2.8, 7.5)
..                 controls  +(180:1.2cm) and +(180:0.5cm)
.. cycle;
\draw (-2.9,6.3) node[inner sep=0.8pt, circle, fill, label={[shift={(0.2,-0.45)}]{\scriptsize 00}}] {} -- (-2.9,7) node[inner sep=0.8pt, circle, fill, label={[shift={(0.2,-0.05)}]{\scriptsize 10}}] {};

\path[fill=gray!40] (0.1,5.8) ..
                 controls +(5:0.5cm) and +(-5:1cm) 
.. (0.1, 7.5)
..                 controls  +(180:1.2cm) and +(180:0.5cm)
.. cycle;
\draw (0,6.3) node[inner sep=0.8pt, circle, fill, label={[shift={(0.2,-0.45)}]{\scriptsize 00}}] {} -- (0,7) node[inner sep=0.8pt, circle, fill, label={[shift={(0.2,-0.05)}]{\scriptsize 11}}] {};

\path[fill=gray!40] (3,5.8) ..
                 controls +(5:0.5cm) and +(-5:1cm) 
.. (3, 7.5)
..                 controls  +(180:1.2cm) and +(180:0.5cm)
.. cycle;
\draw (2.9,6.3) node[inner sep=0.8pt, circle, fill, label={[shift={(0.2,-0.45)}]{\scriptsize 00}}] {} -- (2.9,7) node[inner sep=0.8pt, circle, fill, label={[shift={(0.2,-0.05)}]{\scriptsize 01}}] {};

\path[fill=gray!40] (5.9,5.8) ..
                 controls +(5:0.5cm) and +(-5:1cm) 
.. (5.9, 7.5)
..                 controls  +(180:1.2cm) and +(180:0.5cm)
.. cycle;
\draw (5.8,6.3) node[inner sep=0.8pt, circle, fill, label={[shift={(0.2,-0.45)}]{\scriptsize 01}}] {} -- (5.8,7) node[inner sep=0.8pt, circle, fill, label={[shift={(0.2,-0.05)}]{\scriptsize 11}}] {};


\path[fill=gray!40] (-3.3,4.6) ..
                 controls(-1.7,4.6) and  (-2.4, 2.3)
.. (-3.3, 2.2)
..                 controls(-4.2, 2.3) and (-5,4.6)
.. cycle;
\draw (-2.9,3.7) node[inner sep=0.8pt, circle, fill, label={[shift={(0.1,0)}]{\scriptsize 11}}] {} -- (-3.3,3) node[inner sep=0.8pt, circle, fill, label=below:{\scriptsize 00}] {}  -- (-3.7,3.7)
 node[inner sep=0.8pt, circle, fill, label={[shift={(-0.1,0)}]{\scriptsize 10}}] {};

\path[fill=gray!40] (0,4.6) ..
                 controls(-1.6,4.6) and  (-0.9, 2.3)
.. (0, 2.2)
..  controls(0.9, 2.3) and (1.6,4.6) 
.. cycle;
\draw (-0.4,3.7) node[inner sep=0.8pt, circle, fill, label={[shift={(-0.1,0)}]{\scriptsize 10}}] {} -- (0,3) node[inner sep=0.8pt, circle, fill, label=below:{\scriptsize 00}] {}  -- (0.4,3.7)
node[inner sep=0.8pt, circle, fill, label={[shift={(0.1,0)}]{\scriptsize 01}}] {};

\path[fill=gray!40] (3.3,4.6) ..
                 controls(1.7,4.6) and  (2.4, 2.3)
.. (3.3, 2.2)
..    controls(4.2, 2.3) and (5,4.6) 
.. cycle;
\draw (2.9,3.7) node[ inner sep=0.8pt, circle, fill, label={[shift={(-0.1,0)}]{\scriptsize 11}}] {} -- (3.3,3) node[inner sep=0.8pt, circle, fill, label=below:{\scriptsize 00}] {}  -- (3.7,3.7)
 node[inner sep=0.8pt, circle, fill, label={[shift={(0.1,0)}]{\scriptsize 01}}] {};

\draw (0,2.2) -- (0,1.45);


\path[fill=gray!40] (-1, 1.4) ..
                 controls +(15:0.5cm) and  +(140:1cm)
.. (1.2, 1)
 ..
                 controls +(-45:0.8cm) and  +(0:0.8cm)
.. (0.1, -0.6)  
 ..
                 controls +(180:1cm) and +(195:1.2cm) 
.. cycle;

\draw (0,0)  -- (0,0.7)
node[pos=0, inner sep=0.8pt, circle, fill, label=below:{\scriptsize 00}] {} node[pos=1, inner sep=0.8pt, circle, fill, label=above:{\scriptsize 11}] {};
\draw (0.7,0.7) node[inner sep=0.8pt, circle, fill, label={[shift={(0.1,0)}]\scriptsize 01}] {} -- (0,0)  -- (-0.7,0.7)
node[inner sep=0.8pt, circle, fill, label={[shift={(-0.1,0)}]{\scriptsize 10}}] {};

\path[fill=gray!80] (-4.2, 1.3) ..
                 controls(-5.5,1.1) and  (-4.7, -0.6)
.. (-4.2, -0.7)
 ..                 controls(-3.2, -0.9) and (-2.7,1.6) 
.. cycle;
\draw (-4.3,-0.4)  -- (-4.3,1)
node[pos=0, inner sep=0.8pt, circle, fill, label={[shift={(0.35,-0.3)}]{\scriptsize 00}}] {} node[pos=0.5, inner sep=0.8pt, circle, fill, label=right:{\scriptsize 10}] {} node[pos=1, inner sep=0.8pt, circle, fill, label={[shift={(0.35,-0.2)}]{\scriptsize 11}}] {};

\path[fill=gray!80] (4.4, 1.3) ..
                 controls(3.1,1.1) and  (3.9, -0.6)
.. (4.4, -0.7)
 ..
                 controls(5.4, -0.9) and (5.9,1.6) 
.. cycle;
\draw (4.3,-0.4)  -- (4.3,1)
node[pos=0, inner sep=0.8pt, circle, fill, label={[shift={(0.35,-0.3)}]{\scriptsize 00}}] {} node[pos=0.5, inner sep=0.8pt, circle, fill, label=right:{\scriptsize 01}] {} node[pos=1, inner sep=0.8pt, circle, fill, label={[shift={(0.35,-0.2)}]{\scriptsize 11}}] {};


\path[fill=gray!80] (-4.2, -1.7) ..
                 controls(-3.3,-1.3) and  (-1.8, -1.5)
.. (-1.2, -1.9)
 ..               controls(-0.3,-2.5) and  (-1.2, -4)
.. (-2.5, -4)  
..
                 controls (-4, -3.9) and (-5.3,-2.2)
.. cycle;

\draw (-2.1,-2.8) node[inner sep=0.8pt, circle, fill, label={[shift={(0.35,-0.2)}]{\scriptsize 01}}] {} -- (-2.5,-3.5)  -- (-3.3,-2.1)
node[pos=0, inner sep=0.8pt, circle, fill, label={[shift={(0.35,-0.3)}]{\scriptsize 00}}] {} node[pos=0.5, inner sep=0.8pt, circle, fill, label=left:{\scriptsize 10}] {} node[pos=1, inner sep=0.8pt, circle, fill, label={[shift={(-0.35,-0.2)}]{\scriptsize 11}}] {};

\path[fill=gray!80] (4.2, -1.7) ..
                 controls(3.3,-1.3) and  (1.8, -1.5)
.. (1.2, -1.9)
 ..               controls(0.3,-2.5) and  (1.2, -4)
.. (2.5, -4)  
..
                 controls (4, -3.9) and (5.3,-2.2)
.. cycle;

\draw (2.1,-2.8) node[inner sep=0.8pt, circle, fill, label={[shift={(-0.35,-0.2)}]{\scriptsize 10}}] {} -- (2.5,-3.5)  -- (3.3,-2.1)
node[pos=0, inner sep=0.8pt, circle, fill, label={[shift={(0.35,-0.3)}]{\scriptsize 00}}] {} node[pos=0.5, inner sep=0.8pt, circle, fill, label=right:{\scriptsize 01}] {} node[pos=1, inner sep=0.8pt, circle, fill, label={[shift={(0.35,-0.2)}]{\scriptsize 11}}] {};


\path[fill=gray!80] (1.5, -4.8) ..
                 controls(0.3,-4.2) and  (-1.3, -4.4)
.. (-1.7, -4.8)
 ..               controls(-2.6,-5.5) and  (-1.7, -7)
.. (-0.1, -7.3)  
..
                 controls +(-5:1.2cm) and +(-30:1.8cm)
.. cycle;
\draw (-1.2,-5.3) node[inner sep=0.8pt, circle, fill, label={[shift={(-0.35,-0.2)}]{\scriptsize 11}}] {}   -- (0,-6.8) node[pos=0.5, inner sep=0.8pt, circle, fill, label=left:{\scriptsize 10}] {} node[pos=1, inner sep=0.8pt, circle, fill, label=below:{\scriptsize 00}] {} -- (1.2,-5.3) node[pos=0.5, inner sep=0.8pt, circle, fill, label=right:{\scriptsize 01}] {} node[pos=1, inner sep=0.8pt, circle, fill, label={[shift={(0.35,-0.2)}]{\scriptsize 11}}] {};


\end{tikzpicture}
\caption{$\mathcal{T}(2)$}
\label{fig:u2}
\end{center}
\end{figure}
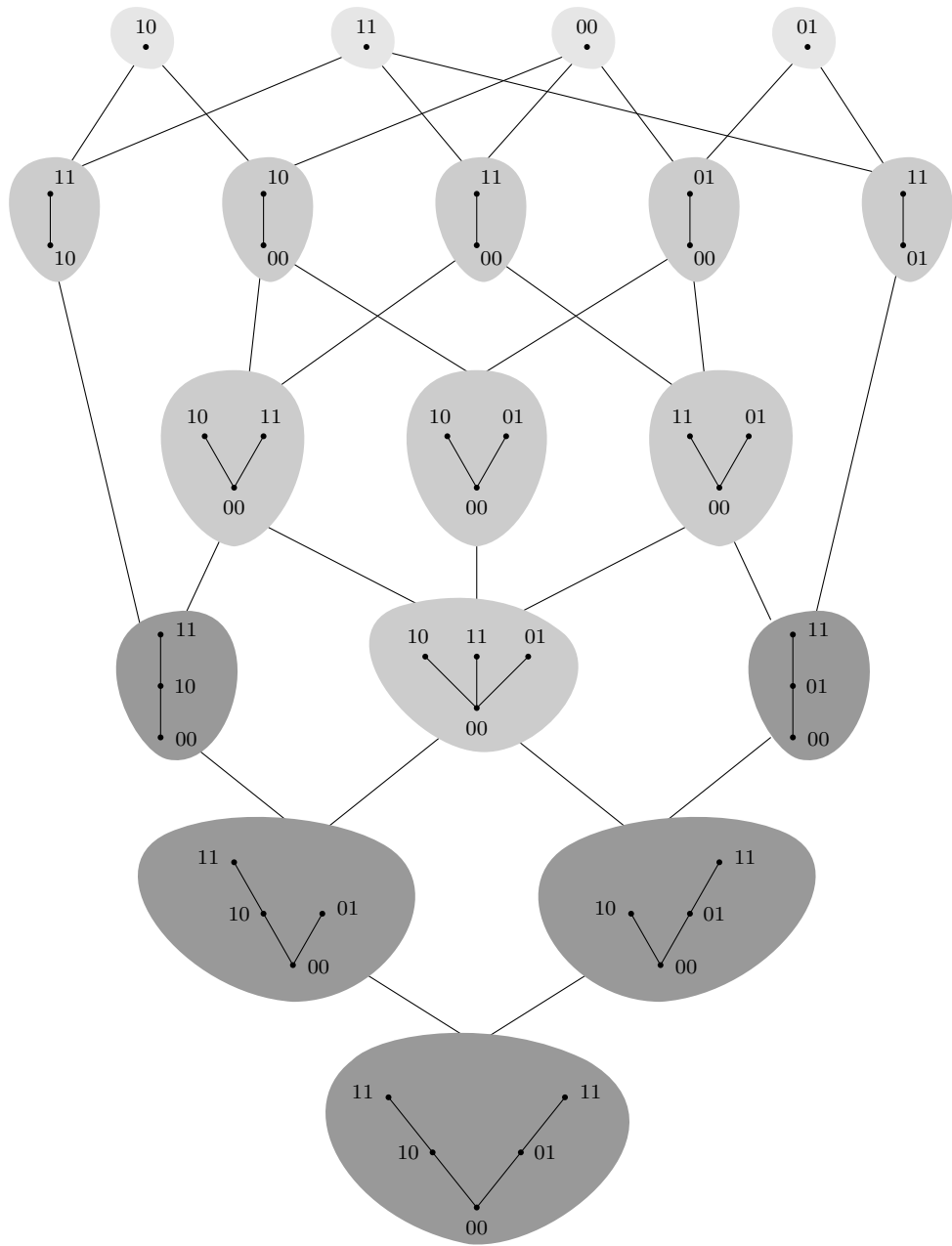


The colors in $\Tn$ are easily seen to be persistent. In the next proposition we verify that $T(n)$ does not contain two $\equiv$-equivalent distinct nodes, which will imply that $\leq$ in $\Tn$ is a partial order as it should be. 
See \Cref{fig:u2} for $\mathcal{T}(2)$ with three layers: The first layer consists of all the single-point trees (indicated in the figure with the lightest shades). The second layer consists of all the nodes with darker shades, while the third layer consists of all the nodes with the darkest shades. 
An easy inductive argument shows that the $m$th layer of $\Tn$ contains trees $T_w$ only of depth $m$. On the other hand, since the relation of $\Tn$ is the $\leq$ relation (instead of the usual generated submodel relation), the depth of a tree $T_w$ regarded as a node in $\Tn$ is often larger, as demonstrated clearly in \Cref{fig:u2}. 



\begin{proposition} \label{prop:properties Tn}
Let $T_w, T_u \in \Tn$. 
\begin{enumerate}
\item If $f: T_w \rightarrow  T_w$ is monotonic and color-preserving, then $f$ is the identity map on $T_w$. \label{a}
\item\label{prop:properties Tn_itm2} If $T_w \neq T_u$, then $T_w \not \equiv T_u$. \label{b}
\end{enumerate}
\end{proposition}

\begin{proof}
(1) We show the claim by induction on $d(T_w)$. If $d(T_w)=1$ the claim is obvious. Now let $d(T_w)>1$ and let $f: T_w \rightarrow T_w$ be a color-preserving monotonic map. Suppose $T_w$ is constructed from the set $X=\{T_{w_1}, \dots T_{w_k}\}$ of pairwise $\leq$-incomparable trees in $\Tn$ by adding a fresh root $w$ below. By construction no element in $X$ has a node of the color of $w$, thus we must have that $f(w)= w$, and $f(w_i)\neq w$ for each $1\leq i\leq k$. The latter implies that  for some $1 \leq j \leq k$, $f[T_{w_i}]$ is a submodel of $T_{w_j}$, or $T_{w_i}\leq T_{w_j}$, which can only happen when $T_{w_i}= T_{w_j}$, as  distinct elements in $X$ are pairwise $\leq$-incomparable. This means that the restricted map $f{\upharpoonright}_{T_{w_i}}{:}~~T_{w_i} \rightarrow T_{w_i}$ is monotonic and color-preserving, and therefore the identity map by induction hypothesis.  We have thus shown that $f$ restricted to all elements in $X$ is the identity map. Together with the fact that $f(w)= w$, we conclude that $f:T_w \rightarrow T_w$ itself is the identity map.


(2) 
Suppose $T_w \equiv T_u$, i.e.~there are color-preserving monotonic maps $f: T_w \rightarrow T_u$ and $g: T_u \rightarrow T_w$. Then $g \circ f : T_w \rightarrow T_w $ and $f \circ g: T_u \rightarrow T_u$ are color-preserving monotonic maps, and thus the identity maps by item \eqref{a}.
So $g$ is a bijective monotonic map with a monotonic inverse. It is a well-known property of partial orders that in this case $T_w$ is isomorphic to $T_u$, and so $T_w= T_u$. 
\end{proof}

As we commented already, every tree $T_w$ in $\Tn$ can be viewed as an unraveling of some generated submodel in $\U(n)$ with each chain strictly decreasing in color; for instance, the right immediate successor of the root of $\mathcal{T}(2)$ in  \Cref{fig:u2} corresponds to the submodel of $\U(2)$ generated by the node $w$ in \Cref{fig:u2ipc}. Meanwhile, the submodel of $\U(2)$ generated by the node $u$, though with every chain strictly decreasing in color, is not present in $\mathcal{T}(2)$, because, e.g., the two subtrees generated by the two immediate successors of $u$ drawn in \Cref{fig:u2ipc} are not $\leq$-incomparable. 



One crucial property of our $n$-universal model $\Tn$ is that it is finite and rooted. 
This is clearly illustrated in \Cref{fig:u2} of $\mathcal{T}(2)$, where the root is the {\em universal $2$-tree}.
Consider the $n$-model $\mathfrak{C}_n=(2^n,\leq,V_n)$ whose domain is the set of all $n$-colors ordered by the partial order $\leq$ of colors, and the color of a node is identical to the node itself. We call  the {\em universal $n$-tree}, denoted by $T_{\mathbf{0}}$,  the unraveling   of $\mathfrak{C}_n$ with root $\mathbf{0}=0\cdots 0$. \todof{changed}

\begin{fact}
$\Tn$ is finite and rooted with root $T_{\mathbf{0}}$.
\end{fact}
\begin{proof}
We first show that $\Tn$ is finite. Since every chain in a tree $T_w$ in $\Tn$ is  strictly decreasing in color, every tree $T_w$  has depth at most $n+1$, which also means that $\Tn$ has at most $n+1$ layers. Each layer is also clearly finite, as essentially there are only finitely many $n$-colors  strictly smaller than one fixed $n$-color.

Next, we show that $T_{\mathbf{0}}$ is in $\Tn$. Clearly, every chain in $T_{\mathbf{0}}$ is strictly decreasing in color, and any two subtrees generated by nodes with the same parent are $\leq$-incomparable (as the two roots have different colors). Thus, an easy inductive argument shows  that every generated subtree in every layer in $T_{\mathbf{0}}$ belongs to $\Tn$, thereby $T_{\mathbf{0}}$ itself is in $\Tn$. 

Finally, we show that $T_{\mathbf{0}}$ is the root of $\Tn$, that is, $T_{\mathbf{0}}\leq T_w$ for every tree $T_w$ in $\Tn$. First note that every two subtrees of $T_{\mathbf{0}}$ generated by two points of the same color $c$ are isomorphic, as they are both the unraveling of the submodel of $\mathfrak{C}_n$ generated by $c$. One can then easily show by induction on the depth of $T_w$ that $T_{c}\leq T_w$ for any subtree $T_{c}$ of $T_{\mathbf{0}}$ generated by a point $c$ with the same color as $w$. Since $T_{\mathbf{0}}\leq T_{c}$, we obtain finally $T_{\mathbf{0}}\leq T_w$. \todof{rewritten and simplified}
%
\end{proof}

As pointed out already, the trees $T_w$ are in general not isomorphic to the submodel of $\Tn$ generated by the node $T_w$. But every node $v$ in a tree $T_w$ does generate a tree $T_v$ from $\Tn$, as the following fact shows. 

\begin{fact}\label{Tw_v_Tn}
If $T_w\in\Tn$ and $v\in T_w$, then the subtree $T_v$ of $T_w$ belongs to $\Tn$. \todof{Made it into a fact, in order to satisfy referee 1. This was not completely clear to me either...}
\end{fact}
\begin{proof}
We prove the fact by induction on the layer in $\Tn$ that $T_w$ belongs to. If $T_w$ is a tree of a single node in the first layer, the fact trivially holds. If $T_w$ is in layer $m+1$, then by the construction of $\Tn$, every immediate successor $v$ of $w$ in $T_w$ generates a tree $T_v$ in $\Tn$. Since $T_v$ is in layer $\leq m$, by induction hypothesis, every node $u$ in $T_v$ generates a tree $T_u$ in $\Tn$.
\end{proof}


Since the nodes in $\Tn$ are models themselves, a formula $\varphi$ can be evaluated at a point $T_w$ of $\Tn$ in two ways: Either  in the model $T_w$ (at its root $w$), or  in the universal model $\Tn$ at the node $T_w$. The next proposition shows that the truth values of $\NNIL$- or $\mon$-formulas $\varphi$ in $n$ variables for both ways of evaluation actually coincide. 
Hereafter we will then use the notation $T_w \models \varphi$ for either $\Tn, T_w \models \varphi$ or $T_w, w \models \varphi$. 

\begin{proposition} \label{prop:properties Tn2}
For any $T_w \in \Tn$, we have $(T_w, w) \simeq_{\NNIL} (\mathcal T (n), T_w) $, and thus $(T_w, w) \simeq_{\mon} (\mathcal T (n), T_w)$ as well.
\end{proposition}
\begin{proof}
Since $\Tn$ and $T_w$ are finite, by \Cref{prop:equivalent}, the result for $\mon$-formulas follows from that for $\NNIL$-formulas. We now prove by induction that $T_w, w\models\varphi$ iff $\mathcal T (n), T_w\models\varphi$ for any $\NNIL$-formula $\varphi$ in $n$ variables in normal form. 

If $\varphi=\bot$, the claim holds trivially. If $\varphi=p$, the claim also holds since $col(w)=col(T_w)$. The induction steps for $\wedge$ and  $\vee$ are easy. We now prove the claim for the case $\varphi= p \rightarrow \psi$.

Suppose $T_w, w \models p \rightarrow \psi$.  To show that  $\Tn, T_w \models p \rightarrow \psi$ let $T_u\in \Tn$ be such that $T_w \leq T_u$ and  $\Tn, T_u \models p$. Then we obtain $T_u, u \models p \rightarrow \psi$ by \Cref{nnil}.  Moreover, since $col(u)=col(T_u)$, we have $T_u, u \models p$, which then implies $T_u, u \models \psi$. Thus, we conclude by induction hypothesis that $\Tn, T_u \models \psi$, as required.

Conversely, suppose $\Tn, T_w \models p\to \psi$. Let $u \in T_w$ be such that $T_w, u \models p$. By \Cref{Tw_v_Tn}, we know that $T_u\in\Tn$. Since $col(u)=col(T_u)$, we have $\Tn,T_u\models p$. As the identity map from $T_u$ into $T_w$ is monotonic and color-preserving, we also have  $T_w\leq T_u$. It then follows that $\Tn, T_u \models \psi$, which, by induction hypothesis, implies that $T_u, u \models \psi$. Again since $T_w\leq T_u$, we obtain $T_w,u\models\psi$ by \Cref{nnil}, and hence $T_w,w\models p\to\psi$.
%
\end{proof}




One can view the trees $T_w$ in $\Tn$ as representatives of $\equiv$-equivalence classes of  $n$-trees, in the sense that every finite tree-like $n$-model is $\equiv$-equivalent to some (unique) tree $T_w$ in $\Tn$. We now prove this important property of $\Tn$.

\begin{lemma}\label{Tn_turth_lem} For every finite  $n$-tree $\mT$, there is a node $T_w$ in $\Tn$ such that 
\begin{enumerate}
\item $T_w\leq_f\mT$ via some surjective color-preserving monotonic map $f$, \todof{shortened this item using the new notation}
\item $T_w$ is isomorphic to a submodel of $\mT$ that has the same root as $\mT$,
\end{enumerate}
\end{lemma}
\begin{proof} We prove the lemma by the number of different colors in $\mT$. If all nodes in $\mT$ have one and the same color $c$, then the tree $T_w$ in $\Tn$ consisting of a single node with the color $c$ clearly satisfies conditions (1) and (2). \todof{corrected the proof, as pointed out by referee 2}



Now assume that the nodes in $\mT$ have at least two different colors. Let $x_1,\dots,x_m$ be the minimal nodes in $\mT$ with a color different from the root $r$, and let ${\mT}_1,\dots,{\mT}_m$ be the subtrees in $\mT$ generated by these nodes respectively. Applying the induction hypothesis to these subtrees we get the corresponding trees $T_{w_1},\dots,T_{w_m}$ in $\Tn$  satisfying conditions (1) and (2). Assume without loss of generality that $T_{w_1},\dots,T_{w_k}$ are the minimal elements among $T_{w_1},\dots,T_{w_m}$ with respect to $\leq$ (and are therefore pairwise incomparable). Let $T_w$ be the tree formed by taking the disjoint union of $T_{w_1},\dots,T_{w_k}$ and adding a root $w$ below with $col(r)$ (which is strictly less than all colors occurring in each $T_{w_i}$). Clearly $T_w$ is a node in $\Tn$. We now verify that $T_w$ satisfies conditions (1) and (2).

To see condition (2), by induction hypothesis each $T_{w_i}$ is isomorphic to a submodel of $\mT_i$ with root $x_i$, and $col(w)=col(r)$. Thus the tree $T_w$ is isomorphic to a submodel of $\mT$ with root $r$.

For condition (1), first for each $1\leq i\leq m$, there is a color-preserving monotonic map $f_i$  from $\mT_i$ onto $T_{w_i}$ given by induction hypothesis. Also, for each $k+1\leq i\leq m$, there is $1\leq j_i\leq k$ such that $T_{w_{j_i}}\leq T_{w_i}$ via a color-preserving monotonic map $g_i: T_{w_i}\to T_{w_{j_i}}$. Now,  we define a map $f\!:\,{\mT}\to T_{w}$ by taking
\[f(x)=\begin{cases}
w&\text{ if }col(x)=col(r);\\
f_i(x)&\text{ if }x\in \mT_i\text{ for some }i\leq k;\\
g_i\circ f_i(x)&\text{ if }x\in \mT_i\text{ for some }k+1\leq i\leq m.
\end{cases}\]
It is easy to see that $f$ is monotonic, color-preserving and onto.
%
%
%
\end{proof}

\begin{theorem} 
\label{Tn_turth_thm}
For every finite $n$-tree $\mathfrak T$, there is a unique node $T_w$ in $\mathcal T(n)$ with $\mathfrak T \equiv T_w$, and thus $\mathfrak T\simeq_{\mon}T_w$.

\end{theorem}
\begin{proof} 
 Let $T_w$ be the tree in $\mathcal T(n)$ given by Lemma \ref{Tn_turth_lem}. Then the two conditions in Lemma \ref{Tn_turth_lem} imply immediately $\mT \equiv T_w$. The uniqueness of $T_w$ follows from  \Cref{prop:properties Tn}\eqref{prop:properties Tn_itm2}. 
\end{proof}


Next, we verify that $\Tn$ satisfies the two conditions in \Cref{def:universalmodel} of universal models with respect to $\NNIL$- or  $\mon$-formulas. Since $\NNIL$-formulas are also $\mon$-formulas, we can then conclude that $\Tn$ is a universal model for both $\NNIL$- and $\mon$-formulas. 
Moreover, we  show that $\Tn$ is actually an {\em exact model} for  $\NNIL$- and $\mon$-formulas, which is a universal model with the second condition strengthened to ``for all upsets $U$ of  $\M$ (not necessarily generated by single points), there is $\varphi \in  \Phi$ such that $V(\varphi)= U$" (see~\cite{renardel}).  Note that infinite universal models (such as the $n$-universal models for {\sf IPC}) are in general not exact, whereas $n$-universal models for locally finite fragments of {\sf IPC} often are (\cite{Lex96,renardel}). 

\begin{proposition}\label{lem:universal}\begin{enumerate}
\item\label{lem:universalb} For any $n$-formulas $\varphi,\psi\,{\in}\,\,\mon$, if $\varphi\not\vdash\psi$, then there exists a node in ${\cal T}(n)$ verifying $\varphi$ and falsifying $\psi$. 
\item\label{lem:universald}\label{lem:universalc} For each upset $U$ of $\Tn$, 
 there exists a \NNIL-formula $\beta^+(U)$ such that for each $T_u\in\Tn$, $T_u\models\beta^+(U)$ iff $T_u\in U$. \todof{changed}

In particular, for each node $T_w$ in $\Tn$, there exists a \NNIL-formula $\beta^+(w)$ such that for each $T_u\in\Tn$, $T_u\models\beta^+(w)$ iff $T_w\leq T_u$. 
 
\end{enumerate}
\end{proposition}
\begin{proof}
(1) If $\varphi\not\vdash\psi$, then there is a finite  tree $\mT$ verifying $\varphi$  and refuting $\psi$. 
By  \Cref{Tn_turth_thm}, there is a node $T_w$ in $\mathcal T(n)$ with the same property. 
 
 
(2) Define $\beta^+(U)=\bigwedge\{\beta(v)\mid T_v\in \Tn\text{ and }T_v\notin U\}$. Suppose $T_u\in U$. For any $T_v\in \Tn$ with $T_v\notin U$, we have $T_u\nleq T_v$, which implies, by \Cref{betamono1}, that $T_u\models\beta(v)$. Hence $T_u\models\beta^+(U)$.

Conversely, suppose $T_u\notin U$. Then $\beta(u)$ is a conjunct in $\beta^+(U)$, and since $T_u\not\models\beta(u)$ by \Cref{lem:selfrefutetation}, we obtain $T_u\not\models\beta^+(w)$.
%
%
%
%
\end{proof}

Now that we have proved all these concrete properties of $\Tn$ 
let us remark
that there are alternative ways to see the universal model $\Tn$. For instance, since $\Tn$ and $\mathcal{U}(n)$ are constructed basically by  the same strategy (see Definitions \ref{def:unipc} and \ref{univnn}), one can view $\Tn$ as the set of those unraveled elements $T$ of $\mathcal{U}(n)$ with only strictly color-decreasing chains such that for every node $w$ in $T$, the trees $T_{w_i}$ generated by immediate successors $w_i$ of $w$ are pairwise $\leq$-incomparable. Indeed, it is easy to see by induction that each such unraveled element $T$ of $\mathcal{U}(n)$ belongs to $\Tn$, and conversely each element in $\Tn$ is  an unraveled element $T$ of $\mathcal{U}(n)$ of the described form. \todof{Reformulated a little the definition, Now I think this fact is indeed clear, so I didn't really add the detailed proof (as there's nothing to prove then).}\todof{And deleted the other alternative definition, as you suggested.}


\subsection{Consequences of the universal model}

Having seen that $\Tn$ is the universal model for both  $\NNIL$- and $\mon$-formulas, we now derive in the following three corollaries an even closer relationship of $\NNIL$-formulas with $\mon$-formulas as well as the class $\textsf{B}$  of $\NNIL$-subframe formulas we defined earlier. First, we conclude that \NNIL-formulas are exactly the ones reflected by color-preserving monotonic maps, a result essentially already following from \cite{VNNIL}.

\begin{corollary}
\label{cor:betaaxiomatization}
For any $\mon$-formula $\varphi$, there is a finite set $\textsf{B}_\varphi\subseteq \textsf{B}$ of \NNIL-subframe formulas such that $\vdash \varphi \leftrightarrow \bigwedge_{\beta(w) \in \textsf{B}_\varphi} \beta(w)$. In particular, every $\mon$-formula is equivalent to a \NNIL-formula, and \NNIL-formulas are (up to equivalence) exactly the formulas  reflected by color-preserving monotonic maps.
\end{corollary}

\begin{proof} 
Consider the upset $V(\varphi)$ in the universal model $\mathcal T(n)$. By (the proof of) \Cref{lem:universal}\eqref{lem:universald} there is a finite set $\textsf{B}_\varphi\subseteq \textsf{B}$ of \NNIL-subframe formulas such that  $T_u \models \bigwedge_{\beta(w) \in \textsf{B}_\varphi} \beta(w)$ iff $T_u \models \varphi$ for every $T_u \in \mathcal T(n)$. Thus $\vdash_{\mathsf{IPC}} \varphi \leftrightarrow \bigwedge_{\beta(w) \in \textsf{B}_\varphi} \beta(w)$ by the property of universal model. 
\end{proof}

Next we infer that the intermediate logics axiomatized by $\textsf{B}$-, $\mon$-, or $\NNIL$-formulas coincide. Therefore, by \Cref{subframe_log_thm}, formulas from one and any of the three classes are sufficient to axiomatize all subframe logics. 
\begin{corollary}\label{BmonNNIL-logics} For an intermediate logic $L$, 
the following are equivalent:
\begin{enumerate}
\item\label{B}  $L$ is axiomatizable by $\textsf{B}$-formulas;
\item $L$ is axiomatizable by $\NNIL$-formulas;
\item $L$ is axiomatizable by $\mon$-formulas;
\item $L$ is a subframe logic.
\end{enumerate}
\end{corollary}
\begin{proof}
By \Cref{subframe_log_thm}, (4) is equivalent to (2).
 Since $\mathsf{B} \subseteq \mathsf{NNIL} \subseteq \mon$, the implications $(3) \Rightarrow (2) \Rightarrow (1)$ are obvious. By \Cref{cor:betaaxiomatization} every $\mon$-formula is equivalent to a set of $\textsf{B}$-formulas, thus $(1)$ implies $(3)$ follows.
\end{proof}

This corollary allows us to complete the discussion from the ending of the previous section and
to conclude surprisingly easily to the well-known fact that subframe logics $L$ are {\em canonical}, i.e., the underlying Kripke frames of the canonical models of $L$ are $L$-frames. 
\todod{I changed the text a little, added the definition in a footnote (for referee 1 who may have assumed that canonical formula and canonical model have something to do with each other, but they don't) and added a line in the proof below.} \todof{I would prefer to move the footnote to the main text. Changed.}

\begin{corollary}
\label{cor:canonical}
The  class of frames of any subframe logic  is closed under taking substructures. In particular, subframe logics are canonical.  
\end{corollary}
\begin{proof}
The first part follows immediately from \Cref{cor:substructures} and the equivalence of items (1) and (4) in \Cref{BmonNNIL-logics}.\todod{ The (1) here is to make clear for the referee that the B-formulas are the point}\todof{I've added also (4) and rephrased it a little (if you read carefully what Cor. 4.13 says)}

For any subframe logic $L$, consider its canonical model. The descriptive frame of its definable upsets validates $L$. Since the underlying Kripke frame of this descriptive $L$-frame is obviously a substructure, it validates $L$, i.e., $L$ is canonical.  
\end{proof}

Finally, let us conclude our discussion on $n$-universal models $\Tn$ for \NNIL-formulas by showing that they are isomorphic to $n$-canonical models for $\NNIL$-formulas.
Let $\NNIL_n$ denote the class of all \NNIL-formulas in $n$ variables, and note that up to equivalence this class is finite since $\NNIL$-formulas are locally finite. 
Recall that the $n$-canonical mode $\M_{\NNIL_n}$ \todof{changed notation} 
for  $\NNIL_n$-formulas is constructed in the standard manner (as in the case of full {\sf IPC}):
Elements in $\M_{\NNIL_n}$ are the {\em consistent theories} of $\NNIL_n$-formulas (i.e., sets $\Gamma$ of $\NNIL_n$-formulas such that $\Gamma\nvdash\bot$, and $\Gamma\vdash\varphi$ implies $\varphi\in \Gamma$ for all $\varphi\in\NNIL_n$) with the {\em disjunction property} (i.e., $\varphi\vee\psi\in\Gamma$ implies either $\varphi\in\Gamma$ or $\psi\in\Gamma$), and the ordering is the set inclusion relation $\subseteq$. Observe that since \NNIL-formulas are locally finite, a theory $\Gamma$ of $\NNIL_n$-formulas is (modulo equivalence) a finite set $\Gamma_0$ or a single formula $\varphi=\bigwedge\Gamma_0$. \todof{added}
Elements of  $\M_{\NNIL_n}$ are thus actually the theories generated by the (single and consistent) $\NNIL_n$-formulas $\varphi$ with the {\em disjunction property} (i.e., $\varphi\vdash\psi\vee\chi$ implies $\varphi\vdash\psi$ or $\varphi\vdash\chi$). 
The $\NNIL_n$-formulas axiomatizing the theories in $\M_{\NNIL_n}$ are in fact the $\beta^+(w)$ formulas defined in \Cref{lem:universal}(\ref{lem:universalc}), as will follow from the next theorem.

\begin{theorem}\label{canonical_universal}
 $\M_{\NNIL_n}$ (with relation $\subseteq$)  is isomorphic to $\Tn$ (with relation $\leq$).
\end{theorem}
\begin{proof}
Define a function $f$ from $\Tn$ into $\M_{\NNIL_n}$ by taking
\[f(T_w)=\mathsf{Th}(T_w)=\{\varphi \in \NNIL_n \mid  T_w \models \varphi \}.\]
Each $\mathsf{Th}(T_w)$ clearly has the disjunction property, and thus $f$ is well-defined. Also, obviously $f$ preserves colors. By property (ii) of the universal model (\Cref{def:universalmodel}), the theories of two distinct points in $\Tn$ differ, and thus $f$ is injective.

To see that $f$ is onto take any $\Phi \in \M_{\NNIL_n}$, and let $\varphi$ axiomatize $\Phi$. Let $\psi_1,\dots,\psi_k$ contain a member of each equivalence class in  $\NNIL_n\setminus \Phi$. Clearly $\Phi\nvdash\psi_i$ for all $i$, as $\Phi$ is a theory.  It then follows that $\Phi\nvdash\bigvee_i\psi_i$, as $\Phi$ has the disjunction property. Now, by property (i) of the universal model (or \Cref{lem:universal}(1)), there is some node $T_w$ in $\Tn$ verifying $\varphi$ and refuting $ \bigvee_i\psi_i$. This means that $\Phi=\mathsf{Th}(T_w)=f(T_w)$, as desired.


It remains to show that $f$ is two ways order-preserving. Clearly $T_w  \leq T_u$ implies $\mathsf{Th}(T_w) \subseteq \mathsf{Th}(T_u)$. Conversely, if $T_w  \not \leq T_u$, by \Cref{betamono1} we have $T_w \models \beta(u)$. Meanwhile, by \Cref{lem:selfrefutetation} we have $T_u \not\models \beta(u)$. Hence $\mathsf{Th}(T_w) \not \subseteq \mathsf{Th}(T_u)$. 
\end{proof}

\section{Finite color-preserving submodels and the finite model property}\label{sec:fmp}
\label{sec:color-preservingsubmodels}

In this section, we give an alternative and direct proof that logics axiomatized by $\NNIL$-formulas (i.e., all subframe logics) have the finite model property. Our proof  is quite different from the previous proofs like the one of Theorem 11.20 in~\cite{CZ97}, which uses canonical formulas. We will introduce a procedure to reduce infinite models to finite submodels in a color-preserving manner. In case the infinite model is tree-like the finite reduced model is indistinguishable from it by $\mon$- or $\NNIL$-formulas. Our central argument then relies heavily on 
the result that logics axiomatized by $\NNIL$-formulas are preserved under substructures (\Cref{BmonNNIL-logics,cor:canonical}). 
In the proof we also make use of a  folklore frame-normal form of {\sf IPC}-formulas, for which we include in the paper a direct semantic proof. 

Let us start by defining the notion of color-preserving submodel that will play an important role in our construction.

\begin{definition} A submodel $\N=(W_0,R,V)$ of a model $\M=(W,R,V)$ is said to be \emph{color-preserving}, denoted $\N\subseteq_c\M$, if, for any $w\in W_0$ and $u\in W$, $wRu$ implies that there exists $v\in W_0$ such that $wRv$ and $col(v)=col(u)$. 
\end{definition}

The above notion
is related to the concept of \textit{relatively open} introduced by Ghilardi in~\cite{Ghi92}. A function $f : \N \to \M$ is relatively open with respect to $g$ (also $g$\textit{-open}) if, whenever $f(w) R u$, there exists $u'$ in $\N$ such that $w R u'$ and $g(f(u')) = g(u)$.
Observe that $\N\subseteq_c\M$
iff the inclusion map $i : \N \to \M$ is $col$-open. \todof{changed a little the phrasing} It would be possible to reformulate proofs in this section, especially the proof of Theorem \ref{fmp}, using the properties of relatively open maps.\todof{added}\todod{reformulated slightly}

In the next lemma we prove some basic facts concerning color-preserving submodels.

\begin{lemma}\label{color-pres_pm} 
\begin{enumerate}
\item If $\f{M}_0\subseteq_c\M_1$ and $\M_1\subseteq_c\M_2$, then $\M_0\subseteq_c\M_2$. 
  \item Let $\alpha$ be a p-morphism from $\M$ into  $\M'$. If $\N\subseteq_c\M$, then $\alpha[\N]\subseteq_c\M'$, where $\alpha[\N]$ is the image of $\N$ under $\alpha$.
\end{enumerate}
\end{lemma}
\begin{proof}
The proofs for both items are standard. We only give the proof of item (2). Let $\M=(W,R,V)$, $\M'=(W',R',V')$ and $\N=(W_0,R,V)$.
Let $w\in W_0$ and $u'\in W'$ be such that $\alpha(w)R'u'$. Since $\alpha$ is a p-morphism, there exists $u\in W$ such that $\alpha(u)=u'$ and $wRu$. By assumption, $\N$ is a color-preserving submodel of $\M$, thus there exists $v\in W_0$ such that $wRv$ and $col(v)=col(u)$. It follows that for $\alpha(v)$ in $\alpha[\N]$, we have  $\alpha(w)R'\alpha(v)$ and $col(\alpha(v))=col(v)=col(u)=col(\alpha(u))=col(u')$, as required. 
\end{proof}

We now proceed to reduce an infinite model to a finite one in a color-preserving manner in a number of steps. In the first two steps we do this with trees, first reducing the depth of a tree-like model to finite, and in the next step the width. In the third step, in 
Theorem~\ref{color_fm}, we adapt this to infinite models in general, connect it with color-preserving monotonic maps, and show that validity of {\sf MR}- and \NNIL-formulas is preserved by the operation.\todod{Added some sentences for the referee}\todof{It is fine, if this is what (s)he wants. Removed ``the basic", do not understand it...}\todod{added a tiny bit more}


\begin{lemma}\label{colorreduction}
Every tree-like $n$-model $\M$ has a tree-like color-preserving submodel $\N$ of finite depth with the same root. 
\end{lemma}
\begin{proof}
Let $\M=(W,R,V)$ and let $r$ be the root. Since $\M$ is a tree, every node $w$ above $r$ has a unique  immediate predecessor that we denote by $w_0$. Let $\N$ be the submodel of $\M$ on the set 
\[W_0=\{r\}\cup\{w\in W\mid col(w_0)<col(w)\}.\]
The model $\N$ has finite depth since all chains in $\N$ are strictly increasing in color and there are only finitely many $n$-colors. 

It remains to check that $\N$ is a color-preserving submodel of $\M$. For any $w\in W_0$  and $u\in W$ such that $wRu$, since $col(w_0)<col(w)\leq col(u)$, there must exist a node $v$ in the finite set $R(w)\cap R^{-1}(u)$ such that $col(v_0)<col(v)=col(u)$. Clearly, $v\in W_0$ and $wRv$. 
\end{proof}

Next, we prune a tree of finite depth further to one of finite width, and thus obtain a finite tree.

\begin{lemma}
\label{iso_remove}
Every tree-like $n$-model $\M$ of finite depth has a finite tree-like color-preserving submodel $\N$ with the same root.
\end{lemma}
\begin{proof}
Assume that $\M$ is of depth $k$. We inductively select a sequence of color-preserving submodels $\N_{k}\subseteq_c\dots\subseteq_c\N_2\subseteq_c\N_1=\M$ from $\M$ such that in each $\N_i$, there are only finitely many non-isomorphic subtrees generated by every node of depth $i$. This way each layer of $\N_k$ and thus $\N_k$ itself will be finite. 

To construct $\N_2$ consider the nodes $w$ of depth $2$. Each such $w$ has only maximal nodes as its immediate successors. These immediate successors have at most $2^n$ color types, and of each color type we keep only one immediate successor of $w$ and remove all the others. Clearly, after the reduction there are only finitely many non-isomorphic subtrees generated by the nodes $w$ of depth $2$, as these $w$ can have at most $2^n$ distinct colors.


Now assume that $\N_i$ has been defined. We construct $\N_{i+1}$ by deleting some subtrees. Consider an arbitrary node $w$ of depth $i+1$. By induction, there are only finitely many non-isomorphic subtrees generated by the immediate successors of $w$. Of each such isomorphism type we keep only one subtree above $w$ and remove all the others. Clearly, after the reduction there are only finitely many non-isomorphic subtrees generated by the nodes $w$ of depth $i+1$, as, again, these $w$ can have at most $2^n$ distinct colors.

Finally, we verify that  $\N_k$ is a color-preserving submodel of $\M$ by proving that $\N_{i+1}=(W_{i+1},R,V)$ is a color-preserving submodel of $\N_i=(W_i,R,V)$ for each $i$. Suppose $w\in W_{i+1}$, $u\in W_i$ and $wRu$. If $u\in W_{i+1}$, then we are done. Otherwise, $u$ is in a subtree $T$ of $\N_i$ that is missing in $\N_{i+1}$. By the construction there
remains an isomorphic copy of $T$ in $\N_{i+1}$ above $w$ and the node corresponding to $u$ in this isomorphic copy will have the same color as $u$.  
 So we are also done.
\end{proof}

The construction in this lemma is quite close to constructing the bisimulation quotient of $\mathfrak{M}$ but here it is relativized to trees. A similar construction occurs already in~\cite{deJ68}.\todod{Added this remark.}

 \todod{Keep this sentence?}\todof{It can be kept}\todod{Threw it out after all, there is enough before Lemma 5.3}
\begin{theorem}\label{color_fm}
Every rooted $n$-model $\M$ has a finite color-preserving submodel $\N$ with the same root.

In addition, if $\M$ is tree-like, then $\N\leq_f\M$ for some surjective map $f$, 
and so $(\M,r)\simeq_{\mon}(\N,r)$ and $(\M,r)\simeq_{\NNIL}(\N,r)$. \todof{rephrased using the new notation}
\end{theorem}
\begin{proof}
We construct $\N$ in stages. First unravel $\M$ to obtain a tree-like model $\M_{\textsf{t}}$ with the same root. Second, apply Lemma \ref{colorreduction} to $\M_{\textsf{t}}$ to obtain a tree-like color-preserving submodel of finite depth with the same root. Lemma \ref{iso_remove} then gives a finite tree-like color-preserving submodel $\N_0$ with the same root. Then, by  Lemma~\ref{color-pres_pm}(1), $\N_0$ is a color-preserving submodel of $\M_{\textsf{t}}$. Let $\alpha$ be the natural p-morphism from $\M_{\textsf{t}}$ onto $\M$. By \Cref{color-pres_pm}(2), the image $\N=\alpha[\N_0]$  of the finite model $\N_0$ under $\alpha$ is a finite color-preserving submodel of $\M$. Since $\alpha$ maps the root of $\M_{\textsf{t}}$ to the root of $\M$, $\N$ and $\M$ have the same root.

Now, suppose in addition that $\M$ is tree-like. Then $\N$ can be obtained directly from $\M$ by subsequently applying \Cref{colorreduction} and \Cref{iso_remove}. We show that $\mathfrak N$ is also a monotonic image of $\M$. 

Let $\N_1$ be the model of depth $k$ obtained from $\M$ as in \Cref{colorreduction}. For each $w$ in $\M$, by the construction (and using the same notation) there is a predecessor  $w'$ of $w$ in $\N_1$ such that  $col(w'_0)<col(w')= col(w)$. Clearly the map $g$ from $\M$ into $\mathfrak N_1$ defined as $g(w)=w'$ is monotonic, color-preserving and onto.

Let $\N=\N_k\subseteq_c\dots\subseteq_c\N_2\subseteq_c\N_1$ be the sequence of models as constructed in \Cref{iso_remove}. We define maps $g_i: \mathfrak N_i \rightarrow \mathfrak N_{i+1}$ for every $1\leq i\leq k-1$ as follows. The map $g_i$ sends a subtree that is removed in the construction to its isomorphic copy that is kept in $\mathfrak N_{i+1}$. Each $g_i$ is clearly monotonic and onto (and in fact it is a p-morphism). Finally, the composition $g_{k-1}\circ\dots\circ g_1\circ g$ is a color-preserving monotonic map from $\M$ onto $\N$.
%

Lastly, together with the fact that the identity map from the submodel $\N$ into $\M$ is monotonic, we conclude that $(\M,r)\simeq_{\mon}(\N,r)$ and $(\M,r)\simeq_{\NNIL}(\N,r)$ by  definition of the class $\mon$ and \Cref{nnil} for $\NNIL$-formulas.
\end{proof}

As a simple application of the above theorem, we can show that \Cref{Tn_turth_thm} holds in case $\mathfrak T$ is an infinite tree as well, because $\mathfrak T$ can be reduced to a finite tree-like submodel $\mathfrak T_0$ that is a monotonic image of $\mT$, in particular, $\mT_0\equiv\mT$. 

In the rest of the section we will prove the finite model property of logics axiomatized by $\NNIL$-formulas as a consequence of \Cref{color_fm}.
Our argument also uses the fact that each {\sf IPC}-formula $\varphi$ can be brought into a frame-normal form of implication complexity $\leq 2$. This result seems to be more or less folklore, although a closely related form is used in~\cite{Sta79,Jerabek16} where syntactic proofs are given. We give, instead, a semantic proof of this fact in the following.

Let us first define the frame-normal form. Given any formula $\varphi$, for each variable $p$ and constant $\bot$ occurring in $\varphi$ we let $s_p=p$ and $s_\bot=\bot$, and for each compound subformula $\psi$ of $\varphi$ we introduce a fresh variable $s_{\psi}$.  

\begin{definition}
Define inductively formulas $\varphi'_+$ 
as follows:
\begin{itemize}
\item If $\varphi=p$, then define $\varphi'_+=\top$. 
\item If $\varphi=\bot$, then define $\varphi'_+=\top$. 
\item 
If $\varphi=\psi\ast\chi$ for $\ast\in\{\wedge,\vee,\to\}$, then define
\[\varphi'_+=\psi'_+\wedge\chi'_+\wedge \big((s_\psi\ast s_\chi)\leftrightarrow s_{\varphi}\big).
\]
\end{itemize}
Define $\varphi'=\varphi'_+\to s_{\varphi}$. \todof{changed as suggested by referee 2}
\end{definition}

Observe that most conjuncts in $\varphi'_+$ are \NNIL-formulas, except for subformulas of the form $(s_{\psi}\to s_{\chi})\to s_{\psi\to\chi}$. 
We  now show that $\varphi$ and $\varphi'$ are frame-equivalent to each other, and thus the formula $\varphi'$ can be viewed as a normal form for {\sf IPC}-formulas over frames.

\begin{proposition}\label{jankov_form_lm}
For any frame $\F$, we have that $\F\models\varphi\iff\F\models\varphi'$.
\end{proposition}
\begin{proof}
To prove the proposition, we first prove the following claim.

\begin{claim}[ 1]
For any formula $\varphi$, any model $\M$ and any node $w$ in $\M$, we have that
\(\M,w\models\varphi'_+\Longrightarrow \M,w\models \varphi\leftrightarrow s_\varphi.\)
\end{claim}
\begin{proofclaim}[ 1]
We prove the claim by induction on $\varphi$. If $\varphi=p$ or $\bot$, then $s_\varphi=\varphi$  by definition, thus the claim holds trivially.

Suppose $\varphi=\psi\ast\chi$ for $\ast\in\{\wedge,\vee,\to\}$. Assume that $\M,w\models\varphi'_+$, i.e., $\M,w\models \psi'_+\wedge\chi'_+\wedge((s_\psi\ast s_\chi)\leftrightarrow s_\varphi)$. By the induction hypothesis, $\M,w\models \psi\leftrightarrow s_\psi$ and $\M,w\models \chi\leftrightarrow s_\chi$, implying $\M,w\models (\psi\ast\chi)\leftrightarrow (s_\psi\ast s_\chi)$. Since $\M,w\models (s_\psi\ast s_\chi)\leftrightarrow s_\varphi$, we obtain $\M,w\models  (\psi\ast\chi)\leftrightarrow s_\varphi$, as required. 
\end{proofclaim}

Now, to prove the direction ``$\Longrightarrow$'' of the proposition, it suffices to prove that 
\(\M,w\models\varphi\Longrightarrow \M,w\models\varphi'\)
holds for any model $\M$ and any node $w$ in $\M$. Suppose $\M,w\models\varphi$ and $\M,u\models\varphi'_+$ for some successor $u$ of $w$. By Claim 1, $\M,u\models \varphi\leftrightarrow s_\varphi$, thus $\M,u\models s_\varphi$, thereby $\M,w\models\varphi'$.

For the direction ``$\Longleftarrow$'',  suppose $(\F,V),w\not\models\varphi$ for some valuation $V$ on $\F$ and $w$ in $\F$. Let $V'$ be a valuation on $\F$  such that $V'(s_{\psi})=V(\psi)$ for every subformula $\psi$ of $\varphi$. 

\begin{claim}[ 2] $(\F,V')\models \varphi'_+$.
\end{claim}
\begin{proofclaim}[ 2]
We prove the claim by induction on the subformulas $\psi$ of $\varphi$. If $\psi=p$ or $\bot$, then $\psi'_+=\top$ and the claim holds trivially. Suppose $\psi=\theta\ast \chi$ for $\ast\in\{\wedge,\vee,\to\}$. Then $\psi'_+=\theta'_+\wedge\chi'_+\wedge ((s_\theta\ast s_\chi)\leftrightarrow s_\psi)$. By the induction hypothesis, we have that $(\F,V')\models\theta'_+\wedge\chi'_+$. Moreover, by the definition, $V'(s_\theta)=V(\theta)$, $V'(s_\chi)=V(\chi)$ and $V(\psi)=V'(s_{\psi})$, which by a simple inductive argument imply that $V'(s_\theta\ast s_\chi)=V(\theta\ast\chi)=V'(s_\psi)$. Thus $(\F,V')\models (s_\theta\ast s_\chi)\leftrightarrow s_\psi$. 
\end{proofclaim}

To complete the proof we need to show that $(\F,V')\not\models\varphi'$, which can be reduced to showing that $(\F,V'),w\not\models\varphi'_+\to s_\varphi$. By Claim 2, we have that $(\F,V'),w\models \varphi'_+$. It then follows from Claim 1 that $(\F,V'),w\models \varphi\leftrightarrow s_\varphi$. Since $V'$ and $V$ agree on the valuation of all propositional variables occurring in $\varphi$, 
the assumption $(\F,V),w\not\models\varphi$  implies that $(\F,V'),w\not\models\varphi$, which gives $(\F,V'),w\not\models s_\varphi$, as desired.
\end{proof}

%


Finally, we are in a position to prove the finite model property for logics axiomatized by $\NNIL$- or $\mon$-formulas.

\begin{theorem}\label{fmp} 
If $L$ is axiomatized by $\mathsf{NNIL}$- or $\mon$-formulas then $L$ has the finite model property.  
\end{theorem}

\begin{proof}
Assume that $L\,{\nvdash}\, \varphi$ for some formula $\varphi$, and by Proposition~\ref{jankov_form_lm}, we may further assume that $\varphi$ is in the frame-normal form $\varphi'_+\to s$ with $n$ propositional variables. Then $\varphi$ is falsified on an $n$-model $\M$ on a rooted descriptive $L$-frame $\F$. 
Let $\N$ be the finite color-preserving submodel of $\M$ with the same root given by \Cref{color_fm}. 
The underlying frame $\mathfrak G$ of $\mathfrak N$ is obviously a substructure of $\mathfrak M$, and thus $\mathfrak G$ is an $L$-frame by  \Cref{cor:canonical}.

It remains to show that $\N$ falsifies $\varphi=\varphi'_+\to s$. Assume w.l.o.g.\ that the root $v$ of $\M$ makes $\varphi'_+$ true and $s$ false. 
By the construction,  $v$ is also the root of $\N$, and $\N,v\not\vDash s$. It remains to prove that $\N,v\models\varphi'_+$. As pointed out already, most conjuncts in $\varphi'_+$ are \NNIL-formulas, and thus remain true in the submodel $\N$. It is left to check that $v$ makes the formulas of the form $(p\to q)\to r$ true in $\N$. Assuming that $w$ is a node in $\N$ such that $\N,w\not\models r$, we need to show that $\N,w\not\models p\to q$. Now, since $\M,w\models (p\to q)\to r$ and $\M,w\not\models r$, we have $\M,w\not\models p\to q$, so there must exist a successor $u$ of $w$ in $\M$ such that $\M,u\models p$ and $\M,u\not\models q$. Since $\N$ is a color-preserving submodel of $\M$, there is a successor $u_0$ of $w$ in $\N$ such that $\N,u_0\models p$ and $\N,u_0\not\models q$, which implies that $\N,w\not\models p\to q$, as required.
\end{proof}

\section{Open problems 
}

In the above we hope to have brought more clarity to the role of \NNIL-formulas, both in models and in frames. We think this opens up a number of possibilities for future research. We enumerate some of them.


(1). In~\cite{gool} the 
$[\wedge,\to]$-fragment of \textsf{IPC} was studied using finite duality for distributive lattices and universal models leading to results about how the universal model for that fragment fits into the overall universal model of {\sf IPC}, to results about interpolation, and to the relationship of the subframe formulas connected to that fragment with the Jankov-de Jongh formulas. A similar investigation of the \NNIL-fragment seems indicated, and should also throw light on the intriguing relationship between those two fragments. 

(2). A  clear goal for research will be a characterization of the subclass of those modal subframe logics that are closed under arbitrary substructures in the same way that all  intermediate subframe logics are. Such logics obviously exist, is a syntactic characterization too much to hope for?


(3). In~\cite{Fan08} the 2-universal model was used to initiate a study of subframe logics axiomatized by \NNIL-formulas with 2  variables, for example towards the construction of characteristic frames.  This study can be continued and extended to 3 variables using the 3-universal model.


(4). In~\cite{BdJ15} {\sf ONNILLI}-formulas were introduced, which are strongly related to \NNIL-formulas. The universal models for \NNIL-formulas may, either directly be useful for the study of {\sf ONNILLI}-formulas and the stable logics they axiomatize, or indirectly, in the construction of their own universal models.

(5). Construction of the concrete 3-universal model (as far as it goes) with computer assistance may well increase insight in more-variable \NNIL-formulas.

(6). Construction of a syntactically defined class of formulas preserved under cofinal submodels extending \NNIL and study of its properties, and construction of universal models. Presumably such a class should contain the cofinal subframe formulas of~\cite{Nic06,Nic08}.

(7). It seems worthwhile to investigate the concept of universal model in general for locally finite fragments of {\sf IPC} and other logics. 




\section*{Funding}

The third author was supported by grants 308712 and 330525 of the Academy of Finland, and  Research Funds of the University of Helsinki. 

\section*{Acknowledgments}

The authors would like to thank Nick Bezhanishvili for useful discussions on parts of the paper, and two anonymous referees for valuable suggestions. They also thank Silvio Ghilardi for showing them the connection  to the concept of relatively open, which was also pointed out by one of the referees. 

\bibliographystyle{plain}

\end{document}